\documentclass{birkjour}%

\usepackage{amsmath}%
\usepackage{amsfonts}%
\usepackage{amssymb}%
\usepackage{graphicx}
\newtheorem{thm}{Theorem}[section]

\newtheorem{lem}[thm]{Lemma}
\newtheorem{prop}[thm]{Proposition}
\theoremstyle{definition}
\newtheorem{defn}[thm]{Definition}
\theoremstyle{remark}
\newtheorem{rem}[thm]{Remark}

\numberwithin{equation}{section}
\begin{document}

%-------------------------------------------------------------------------
% editorial commands: to be inserted by the editorial office
%
%\firstpage{1} \volume{228} \Copyrightyear{2004} \DOI{003-0001}
%\seriesextraline{This is the Concrete Title of this Book\br H.E. R and S.T.C. W, Eds.}
%\firstpage{1}

\title[Hardy Spaces and Cubature Formulas]
 {Polyharmonic Hardy Spaces\\ on the Complexified Annulus\\  and Error Estimates\\ of Cubature Formulas}

%----------Author 1
\author[Ognyan Kounchev]{Ognyan Kounchev}

\address{%
Acad. G. Bonchev, bl. 3\\
1113 Sofia\\
Bulgaria\\
and IZKS, University of Bonn}

\email{kounchev@math.bas.bg}

\thanks{The first-named author thanks the Alexander von Humboldt Foundation and Project DO-02-275 with Bulgarian NSF}
%----------Author 2
\author{Hermann Render}
\address{School of Mathematical Sciences\br
Belfield\br
Dublin 4\br
Ireland}
\email{hermann.render@ucd.ie}
\subjclass{Primary 65D30, 32A35; Secondary 41A55}

\keywords{Hardy space, Numerical Integration, Cubature formulas, Error estimate}

\date{March 20, 2012}
\dedicatory{To the memory of Professor Werner Hau\ss mann}
%%% ----------------------------------------------------------------------

\begin{abstract}
The present paper has a twofold contribution: first, we introduce a new
concept of Hardy spaces on a multidimensional complexified annular domain which is closely
related to the annulus of the Klein-Dirac quadric important in Conformal
Quantum Field Theory. Secondly, for functions in these Hardy spaces, we 
provide error estimate for the polyharmonic Gau\ss -Jacobi cubature formulas, 
which have been introduced in previous papers. 
\end{abstract}
%%% ----------------------------------------------------------------------
\maketitle
%%% ----------------------------------------------------------------------
%\tableofcontents

%from Paper_ver6  we take the preamle above 
%from Paper_ver9_technical we take the modifications of the body below

\section{Introduction}

In the classical one-dimensional theory of quadrature formulas there are
different approaches to the estimation of the quadrature formulas of Gauss
type. The first approach established already by A. Markov estimates the error
of a quadrature formula for differentiable functions in $C^{N}\left(
I\right)  $ defined on the interval $I$ by means of its $N-$th derivative, cf.
Krylov \cite{krylov}, chapter $7.1$, or Davis \cite[p. 344]{davis}. However
this approach is usually not very practical beyond derivatives of order five,
see \cite{stroud}. The second approach estimates the error for certain classes
of functions $f$ which are \emph{analytic} on some open set $D$ in
$\mathbb{C}$ containing the interval $I,$ cf. \cite{davisRabinowitz}, chapter
$4.6,$ see also \cite{krylov}, chapter $12.2.$

Approximation of integrals in the multivariate case is a much more difficult
task. In Numerical Analysis, instead of quadrature formula the notion of
cubature formula is often used, see \cite{sobolev}, \cite{stroudBook},
\cite{sobolev2} and the recent survey \cite{cools}. In contrast to the
univariate case there is no satisfactory error analysis available in the
multivariate case, cf. \cite{stroudBook}, \cite{bakhvalov}, and part $4$ in
the last Russian edition of the classical monograph \cite{krylov}. In passing
we mention that the area of quadrature formulas for harmonic functions in
which Werner Hau\ss mann has worked, are an interesting multidimensional
alternative and we refer to \cite{armitageGoldstein},
\cite{goldsteinHaussmann} on this topic.

\subsection{Preliminaries}

In \cite{kounchevRenderArkiv} and \cite{kounchevRenderArxiv} we have
introduced a new multivariate cubature formulae $C_{N}\left(  f\right)  $
depending on a parameter $N\in\mathbb{N}$ which approximates the integral
\begin{equation}%
%TCIMACRO{\dint \limits_{A_{a,b}}}%
%BeginExpansion
{\displaystyle\int\limits_{A_{a,b}}}
%EndExpansion
f\left(  x\right)  d\mu\left(  x\right)  \label{eqintegral}%
\end{equation}
for continuous functions $f:A_{a,b}\rightarrow\mathbb{C}$ defined on the
\emph{annular region}
\begin{equation}
A_{a,b}=\left\{  x\in\mathbb{R}^{d}:a<\left\vert x\right\vert <b\right\} ,
\label{Aab}%
\end{equation}
where
$\left\vert x\right\vert =\sqrt{x_{1}^{2}+....+x_{d}^{2}}$ is the euclidean
norm of $x=\left(  x_{1},...,x_{d}\right)  \in\mathbb{R}^{d}.$
In \cite{kounchevRenderArkiv}, \cite{kounchevRenderArxiv}, we have considered
a special class of signed measures $\mu$ with support in $A_{a,b}$, the
so-called \emph{pseudo-positive measures} (see definition below)  however this
formula is easy to extend to almost all signed measure. The exact definition
of the cubature formula $C_{N}\left(  f\right)  $ will be explained in Section
\ref{Scubature}. One aim of this paper is to provide an error analysis for a
class of functions on the annular region $A_{a,b}$ which exhibit a certain
type of analytical behavior.

In order to give the reader the needed mathematical background we have to
recall some terminology. Let
\[
\mathbb{S}^{d-1}:=\left\{  x\in\mathbb{R}^{d}:\left\vert x\right\vert
=1\right\}
\]
be the unit sphere endowed with the rotation invariant measure $d\theta$. We
shall write $x\in\mathbb{R}^{d}$ in spherical coordinates $x=r\theta$ with
$\theta\in\mathbb{S}^{d-1}.$ Let $\mathcal{H}_{k}\left(  \mathbb{R}%
^{d}\right)  $ be the set of all harmonic homogeneous complex-valued
polynomials of degree $k.$ Then $f\in\mathcal{H}_{k}\left(  \mathbb{R}%
^{d}\right)  $ is called a \emph{solid harmonic} and the restriction of $f$ to
$\mathbb{S}^{d-1}$ a \emph{spherical harmonic} of degree $k$ and we set
\begin{equation}
a_{k}:=\dim\mathcal{H}_{k}\left(  \mathbb{R}^{d}\right)  , \label{eqdim}%
\end{equation}
see \cite{steinWeiss}, \cite{seeley}, \cite{askey}, \cite{okbook} for details.
Throughout the paper we shall assume
\begin{equation}
Y_{k,\ell}:\mathbb{R}^{d}\rightarrow\mathbb{R},\ell=1,...,a_{k}, \label{Ykl}%
\end{equation}
is an \emph{orthonormal basis} of $\mathcal{H}_{k}\left(  \mathbb{R}%
^{d}\right)  $ with respect to the scalar product%
\[
\left\langle f,g\right\rangle _{\mathbb{S}^{d-1}}:=\int_{\mathbb{S}^{d-1}%
}f\left(  \theta\right)  \overline{g\left(  \theta\right)  }d\theta.
\]
We shall often use the trivial identity $Y_{k,\ell}\left(  x\right)
=r^{k}Y_{k\ell}\left(  \theta\right)  $ for $x=r\theta.$ The class of
pseudo-positive measures used for our cubature formula $C_{N}\left(  f\right)
$ is now defined in the following way: a signed measure $\mu$ with support in
$A_{a,b}\subset$ $\mathbb{R}^{d}$ is \emph{pseudo-positive with respect to the
orthonormal basis} $Y_{k,\ell},\ell=1,...,a_{k}$, $k\in\mathbb{N}_{0}$ if the
inequality
\begin{equation}
\int_{\mathbb{R}^{d}}h\left(  \left\vert x\right\vert \right)  Y_{k,\ell
}\left(  x\right)  d\mu\left(  x\right)  \geq0 \label{defpspos}%
\end{equation}
holds for every non-negative continuous function $h:\left[  a,b\right]
\rightarrow\left[  0,\infty\right)  $ and for all $k\in\mathbb{N}_{0}$,
$\ell=1,2,...,a_{k}.$ Let us note that every signed measure $d\mu$ with
bounded variation may be represented (non-uniquely) as a difference of two
pseudo-positive measures. We refer to \cite{kounchevRenderArkiv} for
instructive examples of pseudo-positive measures.

The cubature formula $C_{N}\left(  f\right)  $ approximating the integral
(\ref{eqintegral}) is based on the \emph{Laplace--Fourier series }of the
continuous function $f:A_{a,b}\rightarrow\mathbb{C}$, defined by the formal
expansion
\begin{equation}
f\left(  r\theta\right)  =\sum_{k=0}^{\infty}\sum_{\ell=1}^{a_{k}}f_{k,\ell
}\left(  r\right)  Y_{k,\ell}\left(  \theta\right)  \label{frtheta}%
\end{equation}
where the \emph{Laplace--Fourier coefficient }$f_{k,\ell}\left(  r\right)
$\emph{\ }is defined by
\begin{equation}
f_{k,\ell}\left(  r\right)  =\int_{\mathbb{S}^{d-1}}f\left(  r\theta\right)
Y_{k,\ell}\left(  \theta\right)  d\theta\label{eqfourier}%
\end{equation}
for any positive real number $r$ with $a<r<b$ and $a_{k}$ is defined in
(\ref{eqdim}). There is a strong interplay between algebraic and analytic
properties of the function $f$ and those of the Laplace-Fourier coefficients
$f_{k,\ell}$. For example, if $f\left(  x\right)  $ is a polynomial in the
variable $x=\left(  x_{1},...,x_{d}\right)  $ then the Laplace-Fourier
coefficient $f_{k,\ell}$ is of the form $f_{k,\ell}\left(  r\right)
=r^{k}p_{k,\ell}\left(  r^{2}\right)  $ where $p_{k,\ell}$ is a univariate
polynomial, see e.g. in \cite{steinWeiss} or \cite{sobolev}. Hence, the
\emph{Laplace-Fourier} series (\ref{frtheta}) of a polynomial $f\left(
x\right)  $ is equal to
\begin{equation}
f\left(  x\right)  =\sum_{k=0}^{\deg f}\sum_{\ell=1}^{a_{k}}p_{k,\ell
}(\left\vert x\right\vert ^{2})Y_{k,\ell}\left(  x\right)  \label{gauss}%
\end{equation}
where $\deg f$ is the total degree of $f$ and $p_{k,\ell}$ is a univariate
polynomial of degree $\leq\deg f-k.$ This representation is often called the
\emph{Gauss representation.} A similar formula is valid for a much larger
class of functions. Let us recall that a function $f:G\rightarrow\mathbb{C}$
defined on an open set $G$ in $\mathbb{R}^{d}$ is called \emph{polyharmonic of
order} $N$ if $f$ is $2N$ times continuously differentiable and
\begin{equation}
\Delta^{N}u\left(  x\right)  =0 \label{dlNux=0}%
\end{equation}
for all $x\in G$ where $\Delta=\frac{\partial^{2}}{\partial x_{1}^{2}%
}+...+\frac{\partial^{2}}{\partial x_{d}^{2}}$ is the Laplace operator and
$\Delta^{N}$ the $N$-th iterate of $\Delta.$ The theorem of Almansi states
that for a polyharmonic function $f$ of order $N$ defined on the ball
$B_{R}=\left\{  x\in\mathbb{R}^{d}:\left\vert x\right\vert <R\right\}  $ there
exist univariate polynomials $p_{k,\ell}\left(  r\right)  $ of degree $\leq
N-1$ such that
\begin{equation}
f\left(  x\right)  =\sum_{k=0}^{\infty}\sum_{\ell=1}^{a_{k}}p_{k,\ell
}(\left\vert x\right\vert ^{2})Y_{k,\ell}\left(  x\right)  \label{eqAlmansi}%
\end{equation}
where convergence of the sum is uniform on compact subsets of $B_{R},$ see
e.g. \cite{sobolev}, \cite{avanissian}, \cite{aron}.

In this paper an important role is played by the representation of
polyharmonic functions on the annular region $A_{a,b}$ developed by Vekua and
Sobolev (see \cite{vekua}, \cite{sobolev}): It is especially simple if the
dimension $d$ of the euclidean space $\mathbb{R}^{d}$ is \emph{odd. }Then for
any polyharmonic function $f$ of order $N$ defined on the annular region
$A_{a,b}$ there exist univariate $\widetilde{u}_{k,\ell}\left(  \cdot\right)
,$ $\widetilde{\widetilde{u}}_{k,\ell}\left(  \cdot\right)  $ polynomials of
degree $\leq N-1\ $such that
\begin{equation}
u\left(  x\right)  =\sum_{k=0}^{\infty}\sum_{\ell=1}^{a_{k}}\left(
\widetilde{u}_{k,\ell}\left(  \left\vert x\right\vert ^{2}\right)
r^{k}+\widetilde{\widetilde{u}}_{k,\ell}\left(  \left\vert x\right\vert
^{2}\right)  r^{-d-k+2}\right)  Y_{k,\ell}\left(  \theta\right)
\label{uxAlmansiAnnulus}%
\end{equation}
where convergence of the sum is uniform on compact subsets of $A_{a,b}.$ An
equivalent way to express this is to write
\[
u\left(  r\theta\right)  =\sum_{k,\ell}\left(
%TCIMACRO{\dsum _{j=0}^{N-1}}%
%BeginExpansion
{\displaystyle\sum_{j=0}^{N-1}}
%EndExpansion
c_{k,\ell;j}r^{k+2j}+%
%TCIMACRO{\dsum _{j=0}^{N-1}}%
%BeginExpansion
{\displaystyle\sum_{j=0}^{N-1}}
%EndExpansion
d_{k,\ell;j}r^{-d-k+2+2j}\right)  Y_{k,\ell}\left(  \theta\right)  .
\]
In particular, one obtains a representation of the polynomials if all
$d_{k,\ell;j}$ are zero.

\subsection{Complexification of the annulus and the annulus of the Klein-Dirac
quadric}

We want to study analytical extensions of functions $f$ defined on the annular
region using the Laplace-Fourier series. Instead of working with functions $f$
which are \emph{a priori} analytically extendible to a fixed domain $U$ in the
complex space $\mathbb{C}^{d}$ we shall require only that we can extend the
function $x=r\theta\longmapsto f\left(  r\theta\right)  $ to an analytic
function $z\theta\longmapsto f\left(  z\theta\right)  ,$ so we only complexify
the radial variable $r$ to a complex variable $z.$ Then one should expect from
(\ref{eqfourier}) that the Laplace-Fourier coefficient $f_{k,\ell}\left(
r\right)  $ extends to an analytic function of one variable.

To make things more precise let us introduce the basis functions
\begin{equation}
R_{j}\left(  r\right)  :=\left\{
\begin{array}
[c]{c}%
r^{2j+k}\qquad\quad\text{for }j=0,1,...,N-1\\
r^{-d-k+2+2\left(  j-N\right)  }\qquad\text{for }j=N,N+1,...,2N-1.
\end{array}
\right.  \label{RjrAnnulus}%
\end{equation}
By representation (\ref{uxAlmansiAnnulus}) we see that the functions
\[
\left\{  R_{j}\left(  r\right)  Y_{k,\ell}\left(  \theta\right)
,\quad\text{for }k\geq0,\ell=1,2,...,a_{k},\ j=0,1,...,2N-1\right\}
\]
form a basis for the functions polyharmonic of order $N$ in the annulus
$A_{a,b}.$\footnote{The functions $R_{j}$ for $j=0,1,..,N-1$ are called
polyharmonic Kelvin transforms of the functions $R_{N+j},$ cf. \cite{aron}, p.
$6.$}

For every fixed $k$ we introduce the following convenient notation for the set
of indices to be used further:
\begin{equation}
\mathcal{R}_{k}:=\left\{  j\in\mathbb{Z}:j=k+2m,\ \text{or }%
j=-d-k+2+2m\ \text{for an integer }m\geq0\right\}  . \label{RieszSetAnnulus}%
\end{equation}

As said above, the main feature of our approach is to work with a
complexification of the functions $f$ defined in the annulus $A_{a,b}$ of the
form $f\left(  z\theta\right)  .$ For polyharmonic functions $f$
representation (\ref{uxAlmansiAnnulus}) provides a direct description of this
complexification by considering the \textbf{complexified basis} functions
$R_{j}\left(  z\right)  $; we consider the linear space generated by the
functions:
\begin{equation}
V_{k,d}:=\operatorname*{span}\left\{  \left\{  z^{2j+k}\right\}
_{j=0}^{\infty},\ \left\{  z^{-d-k+2+2j}\right\}  _{j=0}^{\infty}\right\}
\label{Vkn}%
\end{equation}

Now in these new notations, a generic element in the space $V_{k,d}$ has the
form $\widetilde{u}_{k,\ell}\left(  z^{2}\right)  z^{k}+\widetilde
{\widetilde{u}}_{k,\ell}\left(  z^{2}\right)  z^{-d-k+2}$ where $\widetilde
{u}_{k,\ell}\left(  \cdot\right)  $ and $\widetilde{\widetilde{u}}_{k,\ell
}\left(  \cdot\right)  $ are polynomials. Obviously, $V_{k,d}\subset
\mathcal{H}\left(  \mathbb{A}_{a,b}\right)  ,$ $V_{k,d}\subset\mathcal{H}%
\left(  \mathbb{A}_{0,b}\right)  ,$ and $V_{k,d}\subset\mathcal{H}\left(
\mathbb{C}\setminus\left\{  0\right\}  \right)  $ where $\mathcal{H}\left(
\Omega\right)  $ denotes the space of all \textbf{analytic functions} in a
domain $\Omega\subset\mathbb{C}.$

Thus, we have come naturally to the next objective of the present research,
the Function theory in the space $\mathbb{C}\times\mathbb{S}^{n-1}.$ The main
interest in the present paper is devoted to the \textbf{complexified}
\textbf{annulus} $\mathcal{A}_{a,b}$ given by
\begin{equation}
\mathcal{A}_{a,b}:=\left\{  \left(  z,\theta\right)  :\ a<\left\vert
z\right\vert <b,\ \theta\in\mathbb{S}^{d-1}\right\}  =\mathbb{A}_{a,b}%
\times\mathbb{S}^{d-1}; \label{AnnulusComplexified}%
\end{equation}
here $\mathbb{A}_{a,b}$ is the annulus in $\mathbb{C}$,
\[
\mathbb{A}_{a,b}:=\left\{  z\in\mathbb{C}:a<\left\vert z\right\vert
<b\right\}  .
\]

Our main novelty will be a multivariate generalization of $H^{2}\left(
\mathbb{A}_{a,b}\right)  $ called \textbf{(polyharmonic) Hardy space on the
annulus }$\mathcal{A}_{a,b}$ to be introduced in Definition
\ref{DHardyAnnulusRn}. We will denote this space by $H_{L}^{2}\left(
\mathcal{A}_{a,b}\right)  $ since it depends on a parameter $L>1.$ The name
\emph{polyharmonic} comes from the fact that  $H_{L}^{2}\left(  \mathcal{A}%
_{a,b}\right)  $ is obtained as a limit of the complexifications of the
polyharmonic functions in the annulus (\ref{uxAlmansiAnnulus}), namely, of the
space:
\begin{equation}
\mathcal{V}:=\operatorname*{span}\left\{  z^{k+2j}Y_{k,\ell}\left(
\theta\right)  ,\ z^{-d-k+2+2j}Y_{k,\ell}\left(  \theta\right)  :\ k,j\geq
0,\ell=1,2,...,a_{k}\right\}  ,\label{V}%
\end{equation}
where $z\in\mathbb{C}$ and $\theta\in\mathbb{S}^{d-1}.$ Clearly, the elements
of $\mathcal{V}$ are finite sums of the type
\begin{equation}
u\left(  z,\theta\right)  =\sum_{k,\ell}\left(  \widetilde{u}_{k,\ell}\left(
z^{2}\right)  z^{k}+\widetilde{\widetilde{u}}_{k,\ell}\left(  z^{2}\right)
z^{-d-k+2}\right)  Y_{k,\ell}\left(  \theta\right)
,\label{uzthetaAlmansiAnnulus}%
\end{equation}
where $\widetilde{u}_{k,\ell}\left(  \cdot\right)  $ and $\widetilde
{\widetilde{u}}_{k,\ell}\left(  \cdot\right)  $ are algebraic polynomials. In
view of the Polyharmonic Paradigm announced in \cite{okbook}, the space
$H_{L}^{2}\left(  \mathcal{A}_{a,b}\right)  $  generalizes the classical Hardy
spaces which are obtained as limits of (Laurent) polynomials, where the degree
of a polynomial is replaced by a degree of polyharmonicity. 

Let us summarize briefly our objectives: The main approach to the
\emph{polyharmonic Hardy space} $H_{L}^{2}\left(  \mathcal{A}_{a,b}\right)  $
is to generate it by the complexified finite sums (\ref{uzthetaAlmansiAnnulus}%
), by introducing an appropriate inner product (and norm). This Hardy space
will be a Hilbert space and we will provide a \emph{Cauchy type kernel}, which
is the analog and a generalization to the Hua-Aronszajn kernel in the ball
(cf. \cite{aron}, p. 125, Corollary $1.1$). Let us note that the last is a
multidimensional generalization of the classical Cauchy kernel $\frac{1}{z}$
from Complex Analysis (various about Cauchy kernels see in \cite{kerzmanStein}%
, \cite{krantzHarmonic}).

\begin{rem}
\label{Rkdq} Our results are closely related to the Function theory on the
\textbf{Klein quadric} (called sometimes Klein-Dirac quadric) given by:\
\begin{equation}
\operatorname{KDQ}:=\left\{  z\theta:\ z\in\mathbb{C},\ \theta\in
\mathbb{S}^{d-1}\right\}  =\mathbb{C}\times\mathbb{S}^{d-1}/\mathbb{Z}%
_{2},\label{KDQ}%
\end{equation}
which may be considered as a subset in $\mathbb{C}^{d}.$ The factor
$\mathbb{Z}_{2}$ means that in the set $\mathbb{C}\times\mathbb{S}^{d-1}$ we
identify the elements $\left(  z,\theta\right)  \approx\left(  -z,-\theta
\right)  $ and the elements $z\theta=\left(  z\theta_{1},z\theta
_{2},...,z\theta_{d}\right)  $ of $\operatorname{KDQ}$ are points in
$\mathbb{C}^{d}.$ The space $\operatorname{KDQ}$ is very important for us
since a genuine counterpart of the present Hardy spaces exists only on  the
ball $\mathbb{D}\times\mathbb{S}^{n-1}/\mathbb{Z}_{2}$ of $\operatorname{KDQ}$
(and not on the ball  $\mathbb{D}\times\mathbb{S}^{n-1}$) which is due to the
fact that the polyharmonic functions in the ball in $\mathbb{R}^{d}$ have
representation of the type (\ref{uxAlmansiAnnulus}) with terms $z^{k+2j}%
Y_{k,\ell}\left(  \theta\right)  =z^{2j}Y_{k,\ell}\left(  z\theta\right)  $
depending on $z\theta.$ In a forthcoming research
\cite{kounchevRenderKleinDirac},  \cite{kounchevRenderBook} we study Hardy
spaces on the ball\ in $\operatorname{KDQ}.$ For the Klein quadric the ball
$\mathbb{D}\times\mathbb{S}^{n-1}/\mathbb{Z}_{2}$ is the compactified
four-dimensional \textbf{Minkowski space-time} $\mathbb{S}^{1}\times
\mathbb{S}^{3}/\mathbb{Z}_{2}.$ It is hoped that the present research provides
a useful framework for CFT by approximating the fields by elements of the
Hardy space in the annulus $\mathcal{A}_{a,b},$ see Remark \ref{RHardyBall} below.

The Klein-Dirac quadric $\operatorname{KDQ}$ defined in (\ref{KDQ})
for arbitrary dimension $d,$ has been originally introduced in a special case
by Felix Klein in his Erlangen program in $1870$, where he put forth his
correspondence between the lines in complex projective 3-space and a general
quadric in projective 5-space. The physical relevance of the quadric and the
relation to the conformal motions of compactified Minkowski space-time had
been exploited by Paul Dirac in $1936$ \cite{dirac}. The Klein-Dirac quadric
plays an important role also in Twistor theory, where it is related to the
complexified compactified Minkowski space \cite{penrose}. Apparently, the term
"Klein-Dirac quadric" for arbitrary dimension $d,$ has been coined by the
theoretical physicist I. Todorov, cf. e.g. \cite{nikolovTodorov},
\cite{nikolovTodorov2}. In these references important aspects of the Function
theory on the ball in $\operatorname{KDQ}$ have been considered in the context
of Conformal Quantum Field Theory (CFT). In the context of CFT Laurent
expansions appear in a natural way as the field functions in the higher
dimensional conformal vertex algebras (using a complex variable
parametrization of compactified Minkowski space); see in particular formula
(4.43) in \cite{nikolovTodorov3}, as well as the references
\cite{nikolovTodorov}, \cite{nikolovTodorov2}.
\end{rem}

The paper is organized as follows: in Section \ref{Sbackground} we recall
background material about the Hardy space $H^{2}\left(  \mathbb{A}%
_{a,b}\right)  .$ In Section \ref{ScomponentH} we define the "component"
subspaces $H^{2,k}\left(  \mathbb{A}_{a,b}\right)  $ of $H^{2}\left(
\mathbb{A}_{a,b}\right)  $ which will be necessary for the multivariate case.
In Section \ref{SHK} we introduce the polyharmonic Hardy space $H_{L}%
^{2}\left(  \mathcal{A}_{a,b}\right)  $ on the complexified annulus
$\mathcal{A}_{a,b}.$ We prove that it is a Hilbert space, a maximum principle,
and infinite-differentiability of the functions in $H_{L}^{2}\left(
\mathcal{A}_{a,b}\right)  .$ In Section \ref{ScauchyComponent} we construct a
Cauchy type kernel for the component spaces $H^{2,k}\left(  \mathbb{A}%
_{a,b}\right)  $ and in Section \ref{ScauchyHK} a Cauchy type kernel for the
space $H_{L}^{2}\left(  \mathcal{A}_{a,b}\right)  .$ In Section
\ref{Scubature} we prove the main result of the paper, about the error
estimate of the polyharmonic Gau\ss -Jacobi cubature formulas introduced in
\cite{kounchevRenderArxiv}, \cite{kounchevRenderArkiv}.

\section{The inner product and norm in the annulus in $\mathbb{C}$
\label{Sbackground}}

Let us recall the classical Hardy space $H^{2}\left(  \mathbb{A}_{a,b}\right)
$ on the annulus $\mathbb{A}_{a,b}$ and its basic properties, cf.
\cite{sarason}, p. $4$; see also \cite{axler}.

Let us remind the usual mean value of a function $f$ defined in $\mathbb{A}%
_{a,b}$ as
\begin{equation}
M_{2}\left(  f;r\right)  :=\left\{  \frac{1}{2\pi}\int_{0}^{2\pi}\left\vert
f\left(  re^{i\varphi}\right)  \right\vert ^{2}d\varphi\right\}  ^{1/2}.
\label{M2}%
\end{equation}

\begin{defn}
The Hardy space $H^{2}\left(  \mathbb{A}_{a,b}\right)  $ is defined as the
space of all analytic functions $f$ in $\mathbb{A}_{a,b}$ by the condition
\begin{equation}
\sup_{a<r<b}M_{2}\left(  f;r\right)  <\infty\label{HardyAnnulusClassical}%
\end{equation}

\end{defn}

\begin{rem}
Recall that by a theorem of Hardy and Riesz the function \linebreak\ $\log
M_{2}\left(  f;r\right)  $ is a convex function of $\log r,$ hence its maximum
is attained either for $r\rightarrow a+$ or for $r\rightarrow b-$ ; cf.
\cite{sarason}, p. $4.$
\end{rem}

The following Theorem contains the basic properties of the space
\linebreak\ $H^{2}\left(  \mathbb{A}_{a,b}\right)  ,$ see \cite{sarason}, p.
$7-9.$

\begin{thm}
\label{TSarason} Every $f\in H^{2}\left(  \mathbb{A}_{a,b}\right)  $ is
decomposed in $f=g+h$ where $g\in H^{2}\left(  \mathbb{D}_{b}\right)  $ and
$h\in H^{2}\left(  \mathbb{C}\setminus\overline{\mathbb{D}_{a}}\right)  ,$
where $\mathbb{D}_{r}$ denotes the disc $\left\{  z\in\mathbb{C}:\left\vert
z\right\vert <r\right\}  .$ (Hence, the functions $g$ and $h$ provide the
limiting properties to $f$ )

Let $f\in H^{2}\left(  \mathbb{A}_{a,b}\right)  $ and $f^{\ast}$ be its
limiting function on $\partial\mathbb{A}_{a,b}$ which we denote by
$f_{a}^{\ast}$ on the circle $a\times\mathbb{S}^{1}\subset\mathbb{C},$ and by
$f_{b}^{\ast}$ on the circle $b\times\mathbb{S}^{1}\subset\mathbb{C}.$

(i) If the \textbf{Laurent expansion} of $f$ is given by
\begin{equation}
f\left(  z\right)  =%
%TCIMACRO{\dsum _{j=-\infty}^{\infty}}%
%BeginExpansion
{\displaystyle\sum_{j=-\infty}^{\infty}}
%EndExpansion
c_{j}z^{j} \label{fzsumcj}%
\end{equation}
then $f_{a}^{\ast},f_{b}^{\ast}\in L_{2}\left(  \mathbb{S}^{1}\right)  $ and
have Fourier series respectively
\begin{equation}%
%TCIMACRO{\dsum _{j}}%
%BeginExpansion
{\displaystyle\sum_{j}}
%EndExpansion
c_{j}a^{j}e^{ij\varphi},\qquad%
%TCIMACRO{\dsum _{j}}%
%BeginExpansion
{\displaystyle\sum_{j}}
%EndExpansion
c_{j}b^{j}e^{ij\varphi}. \label{fafb}%
\end{equation}
Vice versa, to every pair of functions $f_{a}^{\ast},f_{b}^{\ast}\in
L_{2}\left(  \mathbb{S}^{1}\right)  $ having Fourier series (\ref{fafb}) there
corresponds uniquely an $f\in H^{2}\left(  \mathbb{A}_{a,b}\right)  $ given by
(\ref{fzsumcj}).

(ii) The following \textbf{Cauchy formula} holds:\
\[
f\left(  z\right)  =\frac{1}{2\pi}%
%TCIMACRO{\dint _{0}^{2\pi}}%
%BeginExpansion
{\displaystyle\int_{0}^{2\pi}}
%EndExpansion
f_{b}^{\ast}\left(  be^{i\varphi}\right)  \frac{be^{i\varphi}}{\left(
be^{i\varphi}-z\right)  }d\varphi-\frac{1}{2\pi}%
%TCIMACRO{\dint _{0}^{2\pi}}%
%BeginExpansion
{\displaystyle\int_{0}^{2\pi}}
%EndExpansion
f_{a}^{\ast}\left(  ae^{i\varphi}\right)  \frac{ae^{i\varphi}}{\left(
ae^{i\varphi}-z\right)  }d\varphi
\]

(iii) The following \textbf{Maximum principle} holds:%
\[
\left\vert f\left(  z\right)  \right\vert \leq\frac{2\left(  \left\Vert
f_{a}^{\ast}\right\Vert _{L_{2}\left(  \mathbb{S}^{1}\right)  }+\left\Vert
f_{b}^{\ast}\right\Vert _{L_{2}\left(  \mathbb{S}^{1}\right)  }\right)  }%
{\min\left(  1-\left\vert z\right\vert /b,\left\vert z\right\vert /a-1\right)
}\qquad\text{for all }a<\left\vert z\right\vert <b.
\]

\end{thm}

\begin{rem}
\label{RinnerProduct}The space $H^{2}\left(  \mathbb{A}_{a,b}\right)  $ is a
Hilbert space with the inner product
\begin{multline}
\left\Vert f\right\Vert _{H^{2}\left(  \mathbb{A}_{a,b}\right)  }%
^{2}=\left\Vert f_{a}^{\ast}\right\Vert _{L_{2}}^{2}+\left\Vert f_{b}^{\ast
}\right\Vert _{L_{2}}^{2}\label{innerProductHannulus}\\
=\lim_{r\longrightarrow a}\frac{1}{2\pi}%
%TCIMACRO{\dint _{0}^{2\pi}}%
%BeginExpansion
{\displaystyle\int_{0}^{2\pi}}
%EndExpansion
f\left(  re^{i\varphi}\right)  \overline{g\left(  re^{i\varphi}\right)
}d\varphi+\lim_{r\longrightarrow b}\frac{1}{2\pi}%
%TCIMACRO{\dint _{0}^{2\pi}}%
%BeginExpansion
{\displaystyle\int_{0}^{2\pi}}
%EndExpansion
f\left(  re^{i\varphi}\right)  \overline{g\left(  re^{i\varphi}\right)
}d\varphi\nonumber
\end{multline}
cf. \cite{sarason}. Let $f\in H^{2}\left(  \mathbb{A}_{a,b}\right)  $ and have
a Laurent expansion
\[
f\left(  z\right)  =%
%TCIMACRO{\dsum _{j=-\infty}^{\infty}}%
%BeginExpansion
{\displaystyle\sum_{j=-\infty}^{\infty}}
%EndExpansion
c_{j}z^{j}.
\]
Then
\[
\left\Vert f\right\Vert _{H^{2}\left(  \mathbb{A}_{a,b}\right)  }^{2}=%
%TCIMACRO{\dsum _{j=-\infty}^{\infty}}%
%BeginExpansion
{\displaystyle\sum_{j=-\infty}^{\infty}}
%EndExpansion
\left\vert c_{j}\right\vert ^{2}\left\Vert z^{j}\right\Vert _{L_{2}\left(
\partial\mathbb{A}_{a,b}\right)  }^{2}=2\pi%
%TCIMACRO{\dsum _{j=-\infty}^{\infty}}%
%BeginExpansion
{\displaystyle\sum_{j=-\infty}^{\infty}}
%EndExpansion
\left\vert c_{j}\right\vert ^{2}\left(  a^{2j}+b^{2j}\right)  .
\]

\end{rem}

\begin{rem}
In \cite{bergman} S. Bergman computes the kernel for the annulus, cf. also
\cite{krantzGeometricFunction} and \cite{MCCULLOUGH}.
\end{rem}

\begin{rem}
\label{PequivalentNorms} The following follows directly from above: For every
function $f\in\mathcal{H}\left(  \mathbb{A}_{a,b}\right)  $ we have the
inequalities
\[
\left\Vert f\right\Vert _{H^{2}\left(  \mathbb{A}_{a,b}\right)  }\leq
2\sup_{a<r<b}M_{2}\left(  f;r\right)  \leq2\left\Vert f\right\Vert
_{H^{2}\left(  \mathbb{A}_{a,b}\right)  }.
\]

\end{rem}

\section{The component one-dimensional spaces $H^{2,k}$ \label{ScomponentH}}

For us important role will be played by the so-called \textquotedblright
component Hardy spaces\textquotedblright\ $H^{2,k}\left(  \mathbb{A}%
_{a,b}\right)  $ which are relatively simple subspaces of the Hardy space
$H^{2}\left(  \mathbb{A}_{a,b}\right)  ,$ cf. some classical references on
Hardy spaces as \cite{hoffmann}, \cite{rudin}.

\begin{defn}
\label{Dakanalytic} The closure of the linear space of analytic functions
$V_{k,d}$ defined in (\ref{Vkn}), in the Hardy space $H^{2}\left(
\mathbb{A}_{a,b}\right)  $ will be denoted by
\begin{equation}
H^{2,k}=H^{2,k}\left(  \mathbb{A}_{a,b}\right)  \label{H2kAnnulus}%
\end{equation}
and will be called \textbf{component Hardy spaces in the annulus }%
$\mathbb{A}_{a,b}.$
\end{defn}

The properties of the space $H^{2,k}\left(  \mathbb{A}_{a,b}\right)  $ easily
follow from those of the usual Hardy spaces $H^{2}\left(  \mathbb{A}%
_{a,b}\right)  $ and $H^{2}\left(  \mathbb{D}_{b}\right)  .$

\begin{prop}
\label{PH2k} Let $f\in H^{2,k}\left(  \mathbb{A}_{a,b}\right)  .$

1. Then $z^{d+k-2}f\left(  z\right)  \in H^{2}\left(  \mathbb{D}_{b}\right)
.$

2. There exist two uniquely determined functions $f_{1},f_{2}\in H^{2}\left(
\mathbb{D}_{b}\right)  $ such that
\begin{equation}
f\left(  z\right)  =z^{k}f_{1}\left(  z^{2}\right)  +z^{-d-k+2}f_{2}\left(
z^{2}\right)  . \label{ff1f2}%
\end{equation}

3. Let $f^{\ast}$ (having two parts, $f_{a}^{\ast}$ and $f_{b}^{\ast}$ ) be
the limiting function for $f$ on $\partial\mathbb{A}_{a,b}.$ Then the
following expansions hold,
\begin{equation}
f_{a}^{\ast}\left(  ae^{i\varphi}\right)  =\sum_{j\in\mathcal{R}}c_{j}%
a^{j}e^{ij\varphi},\qquad f_{b}^{\ast}\left(  be^{i\varphi}\right)
=\sum_{j\in\mathcal{R}}c_{j}b^{j}e^{ij\varphi},
\label{RieszTypeConditionsAnnulus}%
\end{equation}
where the set $\mathcal{R}$ is defined in (\ref{RieszSetAnnulus}). Vice versa,
for every two functions $f_{a}^{\ast}\left(  ae^{i\varphi}\right)  $ and
$f_{b}^{\ast}\left(  be^{i\varphi}\right)  $ with the above expansions there
exists a function $f\in H^{2,k}\left(  \mathbb{A}_{a,b}\right)  $ having
Laurent expansion (\ref{fzsumcj}) for which $f^{\ast}$ is a limiting one.

4. There exist constants $A,B>0$ such that for every $f\in H^{2,k}\left(
\mathbb{A}_{a,b}\right)  $ holds
\[
A\left\Vert f\right\Vert _{H^{2}\left(  \mathbb{A}_{a,b}\right)  }%
\leq\left\Vert f_{1}\right\Vert _{H^{2}\left(  \mathbb{D}_{b}\right)
}+\left\Vert f_{2}\right\Vert _{H^{2}\left(  \mathbb{D}_{b}\right)  }\leq
B\left\Vert f\right\Vert _{H^{2}\left(  \mathbb{A}_{a,b}\right)  };
\]
cf. \cite{sarason}, p. $5.$
\end{prop}

\section{Polyharmonic Hardy space $H_{L}^{2}\left(  \mathcal{A}_{a,b}\right)
$ on the complexified annulus \label{SHK}}

\subsection{A motivation}

Now we may define the polyharmonic Hardy space $H_{L}^{2}\left(
\mathcal{A}_{a,b}\right)  $ on the complexified annulus $\mathcal{A}_{a,b}$ of
$\mathbb{R}^{d}.$ The most important property of this space will be a Maximum
principle, and a related Cauchy type kernel.

Let us explain the problem of finding a proper norm to define the space
$H_{L}^{2}\left(  \mathcal{A}_{a,b}\right)  .$ The elements of the space
$\mathcal{V},$ which is generating $H_{L}^{2},$ are functions $f\left(
z,\theta\right)  $ having finite expansion
\begin{equation}
f\left(  z,\theta\right)  =%
%TCIMACRO{\dsum _{k,\ell}}%
%BeginExpansion
{\displaystyle\sum_{k,\ell}}
%EndExpansion
f_{k,\ell}\left(  z\right)  Y_{k,\ell}\left(  \theta\right)  =%
%TCIMACRO{\dsum _{k,\ell}}%
%BeginExpansion
{\displaystyle\sum_{k,\ell}}
%EndExpansion
\left(  \widetilde{f}_{k,\ell}\left(  z^{2}\right)  z^{k}+\widetilde
{\widetilde{f}}_{k,\ell}\left(  z^{2}\right)  z^{-d-k+2}\right)  Y_{k,\ell
}\left(  \theta\right)  . \label{fztheta}%
\end{equation}
Since all components of this sum are mutually orthogonal in $\theta
\in\mathbb{S}^{d-1},$ it would be natural to consider as a candidate the norm
\[
\left\Vert f\right\Vert _{\ast}^{2}:=%
%TCIMACRO{\dsum _{k,\ell}}%
%BeginExpansion
{\displaystyle\sum_{k,\ell}}
%EndExpansion
\left\Vert f_{k,\ell}\right\Vert _{H^{2,k}}^{2}.
\]
However it is easy to see that such a norm does not guarantee that a Cauchy
sequence of elements $f_{N}\left(  z\theta\right)  ,$ $N\geq1,$ will converge
uniformly on compacts of $\mathcal{A}_{a,b}.$ Indeed, let us consider the
following sequence,
\begin{equation}
u_{N}\left(  z,\theta\right)  =z^{-d-k_{N}+2+2N}\frac{Y_{k_{N},0}\left(
\theta\right)  }{\sqrt[3]{k_{N}}}=\frac{Y_{k_{N},0}\left(  \theta\right)
}{\sqrt[3]{k_{N}}} \label{uN}%
\end{equation}
with $k_{N}=-d+2+2N,$ for $N\geq d.$ Obviously, $u_{N}\in\mathcal{V}$ and we
have
\[
\left\Vert u_{N}\right\Vert _{\ast}=\frac{1}{\sqrt[3]{k_{N}}}\longrightarrow
0\qquad\text{for }N\longrightarrow\infty.
\]
However, as we see below this sequence is even unbounded (for $d\geq3$ ).

\begin{prop}
For $d\geq3,$ the sequence of functions $u_{N}$ defined in (\ref{uN}) is unbounded.
\end{prop}

\begin{proof}
First of all, recall that for every spherical harmonic $Y_{k,\ell}\left(
\theta\right)  $ of degree $k\geq0$ (see (\ref{Ykl})) we have the estimate
\begin{equation}
\left\vert Y_{k,\ell}\left(  \theta\right)  \right\vert \leq Ck^{\frac{d}%
{2}-1}\qquad\text{for all }\theta\in\mathbb{S}^{d-1}.\label{sphericalEstimate}%
\end{equation}
(cf. \cite{seeley}). Let us consider $d=3.$ We have the following
representation of the spherical harmonic which corresponds to the Legendre
polynomial $P_{k}\left(  \cdot\right)  $, (cf. \cite{tikhonovSamarskii},
Appendix $2,$ part II ):
\[
Y_{k,1}\left(  \theta\right)  =\sqrt{\frac{2k+1}{4\pi}}P_{k}\left(
\cos\vartheta\right)  \qquad\text{for }0\leq\vartheta\leq\pi,\ 0\leq
\varphi\leq2\pi;
\]
here we have $\theta=\left(  \cos\varphi\sin\vartheta,\sin\varphi\sin
\vartheta,\cos\vartheta\right)  \in\mathbb{S}^{2},$ and the usual Legendre
polynomials $P_{k}\left(  \cdot\right)  $. Since $P_{k}\left(  1\right)  =1,$
it follows that the sequence $u_{N}\left(  z,\theta\right)  $ is unbounded at
$\vartheta=0.$
\end{proof}

\begin{rem}
An additional observation about the unwished behavior of the sequence
$u_{N}\left(  z,\theta\right)  $, which is easy to see for $d=2$ (where
$Y_{k,1}=\cos k\varphi$ and $Y_{k,2}=\sin k\varphi$) is that derivatives in
$\theta$ of $u_{N}\left(  z,\theta\right)  $ contain non-negative powers of
$k.$
\end{rem}

For that reason a proper norm has to "tame" the behavior of such sequences
including their derivatives: we will introduce an inner product and norm which
contain weight factors for every component $k$, which \textquotedblright
suppress\textquotedblright\ such unwished behavior and this will guarantee
convergence of Cauchy sequences in $H_{L}^{2}\left(  \mathcal{A}_{a,b}\right)
$ on compact subsets of $\mathcal{A}_{a,b}.$

\begin{rem}
In this context let us remind that a function $f\left(  \theta\right)  $ on
$\mathbb{S}^{d-1}$ is real analytic on $\mathbb{S}^{d-1}$ if its
Laplace-Fourier expansion $f\left(  \theta\right)  =%
%TCIMACRO{\dsum _{k=0}^{\infty}}%
%BeginExpansion
{\displaystyle\sum_{k=0}^{\infty}}
%EndExpansion
Y_{k}\left(  \theta\right)  $ (where $Y_{k}\left(  \theta\right)  =%
%TCIMACRO{\dsum _{\ell=1}^{a_{k}}}%
%BeginExpansion
{\displaystyle\sum_{\ell=1}^{a_{k}}}
%EndExpansion
f_{k,\ell}Y_{k,\ell}\left(  \theta\right)  $ ) satisfies
\[
\left\Vert Y_{k}\left(  \theta\right)  \right\Vert <Ce^{-\eta k}%
\qquad\text{for }k\geq1
\]
for some constants $C>0$ and $\eta>0,$ cf. \cite{sobolev}.
\end{rem}

\subsection{The weighted inner product and norm}

Let us introduce the weight factor to be an arbitrary number
\[
L>1.
\]

\begin{defn}
Let $f:=\left(  f_{k,\ell}\right)  _{k,\ell}$ and $g:=\left(  g_{k,\ell
}\right)  _{k,\ell}$ be two sequences of functions with   $f_{k,\ell}$ $\in
H^{2,k}\left(  \mathbb{A}_{a,b}\right)  $ and $g_{k,\ell}\in H^{2,k}\left(  \mathbb{A}_{a,b} \right)  $ for all
$k=0,1,...,$and $\ell=1,..,a_{k}.$ We define the following inner product
\begin{equation}
\left\langle f,g\right\rangle _{H_{L}^{2}}:=\sum_{k,\ell}\left\langle
f_{k,\ell},g_{k,\ell}\right\rangle _{H^{2}\left(  \mathbb{A}_{a,b}\right)
}L^{2k},\label{scalarHK}%
\end{equation}
where $\left\langle ,\right\rangle _{H^{2}\left(  \mathbb{A}_{a,b}\right)  }$
is the inner product in $H^{2}\left(  \mathbb{A}_{a,b}\right)  .$
\end{defn}

It is easy to see that $\left\langle ,\right\rangle _{H^{2}\left(
\mathbb{A}_{a,b}\right)  }$ satisfies all necessary properties of inner
product on a Hilbert space, and the norm is given by
\begin{align}
\left\Vert f\right\Vert _{H_{L}^{2}}^{2} &  :=\left\Vert f\right\Vert
_{H_{L}^{2}\left(  \mathcal{A}_{a,b}\right)  }^{2}:=\sum_{k,\ell}\left\Vert
f_{k,\ell}\right\Vert _{H^{2}\left(  \mathbb{A}_{a,b}\right)  }^{2}%
L^{2k}\label{HKspace}\\
&  =\left(  \frac{1}{2\pi}\sum_{k,\ell}L^{2k}\int_{0}^{2\pi}\left\vert
f_{k,\ell}^{\ast}\left(  ae^{i\varphi}\right)  \right\vert ^{2}d\varphi
+L^{2k}\int_{0}^{2\pi}\left\vert f_{k,\ell}^{\ast}\left(  be^{i\varphi
}\right)  \right\vert ^{2}d\varphi\right)  ^{1/2}.\nonumber
\end{align}
Let us consider the function
\[
f\left(  z,\theta\right)  =%
%TCIMACRO{\dsum _{k,\ell}}%
%BeginExpansion
{\displaystyle\sum_{k,\ell}}
%EndExpansion
f_{k,\ell}\left(  z\right)  Y_{k,\ell}\left(  \theta\right)
\]
with $f_{k,\ell}\in H^{2,k}\left(  \mathbb{A}_{a,b}\right)  $ and $\left\Vert
f\right\Vert _{H_{L}^{2}}<\infty.$ By applying Proposition
\ref{PequivalentNorms} we see that for every $r\in\left(  a,b\right)  $ holds
\begin{align*}
&
%TCIMACRO{\dint _{0}^{2\pi}}%
%BeginExpansion
{\displaystyle\int_{0}^{2\pi}}
%EndExpansion%
%TCIMACRO{\dint _{\mathbb{S}^{d-1}}}%
%BeginExpansion
{\displaystyle\int_{\mathbb{S}^{d-1}}}
%EndExpansion
\left\vert
%TCIMACRO{\dsum _{k,\ell}}%
%BeginExpansion
{\displaystyle\sum_{k,\ell}}
%EndExpansion
f_{k,\ell}\left(  z\right)  Y_{k,\ell}\left(  \theta\right)  \right\vert
^{2}d\varphi d\theta\\
&  =%
%TCIMACRO{\dsum _{k,\ell}}%
%BeginExpansion
{\displaystyle\sum_{k,\ell}}
%EndExpansion%
%TCIMACRO{\dint _{0}^{2\pi}}%
%BeginExpansion
{\displaystyle\int_{0}^{2\pi}}
%EndExpansion
\left\vert f_{k,\ell}\left(  re^{i\varphi}\right)  \right\vert ^{2}d\varphi\\
&  \leq\left\Vert f\right\Vert _{H_{L}^{2}}^{2},
\end{align*}
which shows that $f$ is a function in $L_{2}\left(  \mathcal{A}_{a,b}\right)
.$ 

Now we may give the following definition.

\begin{defn}
\label{DHardyAnnulusRn} The polyharmonic Hardy space $H_{L}^{2}\left(
\mathcal{A}_{a,b}\right)  $ consists of all functions
\[
f\left(  z,\theta\right)  =%
%TCIMACRO{\dsum _{k,\ell}}%
%BeginExpansion
{\displaystyle\sum_{k,\ell}}
%EndExpansion
f_{k,\ell}\left(  z\right)  Y_{k,\ell}\left(  \theta\right)
\]
with $f_{k,\ell}\in H^{2,k}\left(  \mathbb{A}_{a,b}\right)  $ and satisfying
$\left\Vert f\right\Vert _{H_{L}^{2}}<\infty.$
\end{defn}

Let us remark that this definition is completely analogous to the classical
case where one takes the Laurent series of the type $\sum_{j=-\infty}^{\infty
}a_{j}z^{j}$ which are convergent on compacts in the annulus $\mathbb{A}%
_{a,b}.$

It is easy to see that there are equivalent representations of the functions
in the space $H_{L}^{2}\left(  \mathcal{A}_{a,b}\right)  $.

\begin{prop}
\label{PHKspaces} Every element $f\in H_{L}^{2}\left(  \mathcal{A}%
_{a,b}\right)  $ has the following representation as an infinite series (in
$L^{2}-$sense):
\begin{align}
f\left(  z\theta\right)   &  =\sum_{k,\ell}f_{k,\ell}\left(  z\right)
Y_{k,\ell}\left(  \theta\right) \label{HKLaurent}\\
&  =%
%TCIMACRO{\dsum _{k,\ell}}%
%BeginExpansion
{\displaystyle\sum_{k,\ell}}
%EndExpansion%
%TCIMACRO{\dsum _{j\in\mathcal{R}_{k}}}%
%BeginExpansion
{\displaystyle\sum_{j\in\mathcal{R}_{k}}}
%EndExpansion
f_{k,\ell;j}z^{j}Y_{k,\ell}\left(  \theta\right) \nonumber
\end{align}
where for every index $\left(  k,\ell\right)  $ the function $f_{k,\ell
}\left(  z\right)  \in H^{2,k}\left(  \mathbb{A}_{a,b}\right)  ;$ the set
$\mathcal{R}_{k}$ was introduced in (\ref{RieszSetAnnulus}). The norm of $f$
is given by
\begin{align*}
\left\Vert f\right\Vert _{H_{L}^{2}}^{2}  &  =%
%TCIMACRO{\dsum _{k,\ell}}%
%BeginExpansion
{\displaystyle\sum_{k,\ell}}
%EndExpansion
\left(
%TCIMACRO{\dsum _{j\in\mathcal{R}_{k}}}%
%BeginExpansion
{\displaystyle\sum_{j\in\mathcal{R}_{k}}}
%EndExpansion
\left\vert f_{k,\ell;j}\right\vert ^{2}\left\Vert z^{j}\right\Vert
_{H^{2}\left(  \mathbb{A}_{a,b}\right)  }^{2}\right)  L^{2k}\\
&  =2\pi%
%TCIMACRO{\dsum _{k,\ell}}%
%BeginExpansion
{\displaystyle\sum_{k,\ell}}
%EndExpansion
\left(
%TCIMACRO{\dsum _{j\in\mathcal{R}_{k}}}%
%BeginExpansion
{\displaystyle\sum_{j\in\mathcal{R}_{k}}}
%EndExpansion
\left\vert f_{k,\ell;j}\right\vert ^{2}\left(  a^{2j}+b^{2j}\right)  \right)
L^{2k}.
\end{align*}

\end{prop}

The proof is straightforward.

We prove the following fundamental result which contains as its first item that
the space  $H_{L}^{2}\left(  \mathcal{A}_{a,b}\right) $ is a Hilbert one.

\begin{thm}
\label{THKisHilbert} 1. The space $H_{L}^{2}\left(  \mathcal{A}_{a,b}\right)
$ is complete.

2. The following Maximum principle is true: for every compact $K\subset
\mathcal{A}_{a,b}$ there exists a constant $C>0$ such that
\[
\left\vert f\left(  z,\theta\right)  \right\vert \leq\frac{C}{\min\left(
1-\left\vert z\right\vert /b,\left\vert z\right\vert /a-1\right)  }\left\Vert
f\right\Vert _{H_{L}^{2}\left(  \mathcal{A}_{a,b}\right)  }\qquad\text{for
}\left(  z,\theta\right)  \in K.
\]

3. If $f\in H_{L}^{2}\left(  \mathcal{A}_{a,b}\right)  $ then $f\in C^{\infty
}\left(  \mathcal{A}_{a,b}\right)  .$
\end{thm}

%

%TCIMACRO{\TeXButton{Proof}{\proof} }%
%BeginExpansion
\proof
%EndExpansion
For the proof of item 1), let $\left\{  f^{n}\right\}  _{n\geq1}$ be a Cauchy
sequence in $H_{L}^{2}\left(  \mathcal{A}_{a,b}\right)  .$ Then it is clear
that for every index $\left(  k,\ell\right)  $ the sequence of functions
$\left\{  f_{k,\ell}^{n}\left(  z\right)  \right\}  _{n\geq1}$ is a Cauchy
sequence in $H^{2}\left(  \mathbb{A}_{a,b}\right)  $ (even in $H^{2,k}\left(
\mathbb{A}_{a,b}\right)  $ ), and as such has a limit $g_{k,\ell}\left(
z\right)  \in H^{2,k}\left(  \mathbb{A}_{a,b}\right)  .$

We have to show that the function defined by
\[
g\left(  z,\theta\right)  =%
%TCIMACRO{\dsum _{k,\ell}}%
%BeginExpansion
{\displaystyle\sum_{k,\ell}}
%EndExpansion
g_{k,\ell}\left(  z\right)  Y_{k,\ell}\left(  \theta\right)
\]
belongs to $H_{L}^{2}\left(  \mathcal{A}_{a,b}\right)  $ and is the limit of
$\left\{  f^{n}\right\}  _{n\geq1}$ there. Let us fix some $\varepsilon>0.$
Let us take some $k_{1}\geq1.$ By the Cauchy sequence, we find $n_{1}\geq1$
such that for all $n,m\geq n_{1}$ we have
\[%
%TCIMACRO{\dsum _{k=0}^{k_{1}}}%
%BeginExpansion
{\displaystyle\sum_{k=0}^{k_{1}}}
%EndExpansion%
%TCIMACRO{\dsum _{\ell=1}^{a_{k}}}%
%BeginExpansion
{\displaystyle\sum_{\ell=1}^{a_{k}}}
%EndExpansion
\left\Vert f_{k,\ell}^{n}-f_{k,\ell}^{m}\right\Vert _{H^{2}\left(
\mathbb{A}_{a,b}\right)  }^{2}L^{2k}\leq\left\Vert f^{n}-f^{m}\right\Vert
_{H_{L}^{2}\left(  \mathcal{A}_{a,b}\right)  }^{2}<\varepsilon^{2}.
\]
Taking the limit for $m\longrightarrow\infty$ we obtain
\[%
%TCIMACRO{\dsum _{k=0}^{k_{1}}}%
%BeginExpansion
{\displaystyle\sum_{k=0}^{k_{1}}}
%EndExpansion%
%TCIMACRO{\dsum _{\ell=1}^{a_{k}}}%
%BeginExpansion
{\displaystyle\sum_{\ell=1}^{a_{k}}}
%EndExpansion
\left\Vert f_{k,\ell}^{n}-g_{k,\ell}\right\Vert _{H^{2}\left(  \mathbb{A}%
_{a,b}\right)  }^{2}L^{2k}<\varepsilon^{2}.
\]
Then taking the limit $k_{1}\longrightarrow\infty$ ends the proof.

2). Let us consider an annulus $\mathcal{A}_{a^{\prime},b^{\prime}}$ with
$a<a^{\prime}$ and $b^{\prime}<b$ which satisfies $K\subset\mathcal{A}%
_{a^{\prime},b^{\prime}}.$ For a function $f$ given by (\ref{fztheta}), by the
triangle inequality we obtain for all $\left(  z,\theta\right)  \in K$ the
following inequality:
\[
\left\vert f\left(  z,\theta\right)  \right\vert \leq%
%TCIMACRO{\dsum _{k,\ell}}%
%BeginExpansion
{\displaystyle\sum_{k,\ell}}
%EndExpansion
\left\vert f_{k,\ell}\left(  z\right)  \right\vert \left\vert Y_{k,\ell
}\left(  \theta\right)  \right\vert .
\]
By the maximum principle for the annulus, Theorem \ref{TSarason}, and the
estimate for the spherical harmonics (\ref{sphericalEstimate}) we obtain
\begin{equation}
\left\vert f\left(  z,\theta\right)  \right\vert \leq\frac{2C}{\min\left(
1-\left\vert z\right\vert /b,\left\vert z\right\vert /a-1\right)  }%
%TCIMACRO{\dsum _{k,\ell}}%
%BeginExpansion
{\displaystyle\sum_{k,\ell}}
%EndExpansion
\left\Vert f_{k,\ell}\right\Vert _{H^{2}\left(  \mathbb{A}_{a,b}\right)
}k^{\frac{d}{2}-1}. \label{absfztheta}%
\end{equation}
Further, by applying the Cauchy-Bunyakovsky-Schwarz inequality, for every
$\left(  z,\theta\right)  \in\mathcal{A}_{a^{\prime},b^{\prime}}$ we obtain
\begin{align*}
&  \left\vert f\left(  z,\theta\right)  \right\vert \\
&  \leq\frac{2C}{\min\left(  1-\left\vert z\right\vert /b,\left\vert
z\right\vert /a-1\right)  }%
%TCIMACRO{\dsum _{k,\ell}}%
%BeginExpansion
{\displaystyle\sum_{k,\ell}}
%EndExpansion
L^{k}\left\Vert f_{k,\ell}\right\Vert _{H^{2}\left(  \mathbb{A}_{a^{\prime
},b^{\prime}}\right)  }k^{\frac{d}{2}-1}L^{-k}\\
&  \leq\frac{2C}{\min\left(  1-\left\vert z\right\vert /b,\left\vert
z\right\vert /a-1\right)  }\left(
%TCIMACRO{\dsum _{k,\ell}}%
%BeginExpansion
{\displaystyle\sum_{k,\ell}}
%EndExpansion
L^{2k}\left\Vert f_{k,\ell}\right\Vert _{H^{2}\left(  \mathbb{A}_{a^{\prime
},b^{\prime}}\right)  }^{2}\right)  ^{1/2}\left(
%TCIMACRO{\dsum _{k,\ell}}%
%BeginExpansion
{\displaystyle\sum_{k,\ell}}
%EndExpansion
k^{d-2}L^{-2k}\right)  ^{1/2}.
\end{align*}
Since the constant $a_{k}$ in (\ref{Ykl}) is estimated by $a_{k}=O\left(
k^{d-2}\right)  $ (cf. \cite{steinWeiss}, chapter $4.2$) it follows that for
all $\left(  z,\theta\right)  \in\mathcal{A}_{a^{\prime},b^{\prime}}$ holds
\begin{align*}
\left\vert f\left(  z,\theta\right)  \right\vert  &  \leq\frac{C_{1}}%
{\min\left(  1-\left\vert z\right\vert /b,\left\vert z\right\vert /a-1\right)
}\left\Vert f\right\Vert _{H_{L}^{2}\left(  \mathcal{A}_{a,b}\right)  }%
\times\left(
%TCIMACRO{\dsum _{k=0}^{\infty}}%
%BeginExpansion
{\displaystyle\sum_{k=0}^{\infty}}
%EndExpansion
k^{d-1}k^{d-2}L^{-2k}\right)  ^{1/2}\\
&  \leq\frac{C_{2}}{\min\left(  1-\left\vert z\right\vert /b,\left\vert
z\right\vert /a-1\right)  }\left\Vert f\right\Vert _{H_{L}^{2}\left(
\mathcal{A}_{a,b}\right)  }.
\end{align*}
This ends the proof.

3. As above we consider $\mathcal{A}_{a^{\prime},b^{\prime}}.$ Let
$f\in\mathcal{V}$ have the expansion
\[
f\left(  z,\theta\right)  =%
%TCIMACRO{\dsum _{k,\ell}}%
%BeginExpansion
{\displaystyle\sum_{k,\ell}}
%EndExpansion
f_{k,\ell}\left(  z\right)  Y_{k,\ell}\left(  \theta\right)  .
\]
For $\left(  z,\theta\right)  \in\mathcal{A}_{a^{\prime},b^{\prime}}$ we
obtain the estimate
\[
\left\vert \partial f\left(  z,\theta\right)  /\partial z\right\vert \leq%
%TCIMACRO{\dsum _{k,\ell}}%
%BeginExpansion
{\displaystyle\sum_{k,\ell}}
%EndExpansion
\left\vert f_{k,\ell}^{\prime}\left(  z\right)  \right\vert \left\vert
Y_{k,\ell}\left(  \theta\right)  \right\vert .
\]
By the usual Maximum principle in the annulus $\mathbb{A}_{a,b}$ (which is
equivalent to continuity of point evaluation operator in $H^{2}\left(
\mathbb{A}_{a,b}\right)  ,$ cf. section $17.8$ in \cite{rudin}) we obtain
\[
\left\vert \partial f\left(  z,\theta\right)  /\partial z\right\vert \leq C%
%TCIMACRO{\dsum _{k,\ell}}%
%BeginExpansion
{\displaystyle\sum_{k,\ell}}
%EndExpansion
\left\Vert f_{k,\ell}\right\Vert _{H^{2}\left(  \mathbb{A}_{a,b}\right)
}k^{\frac{d}{2}-1}\qquad\text{for all }\left(  z,\theta\right)  \in
\mathcal{A}_{a^{\prime},b^{\prime}},
\]
and the last is convergent as we have seen above. In a similar way, for some
derivative $D_{\theta}^{\alpha}$ in $\theta,$ where $\alpha$ is a multiindex,
we apply the following inequality
\[
\left\vert D_{\theta}^{\alpha}Y_{k,\ell}\left(  \theta\right)  \right\vert
\leq C_{1}k^{\left\vert \alpha\right\vert +\frac{d-2}{2}}\qquad\text{for
}k\geq1,\ \theta\in\mathbb{S}^{d-1},
\]
cf. \cite{seeley}, p. 120. This ends the proof.%

%TCIMACRO{\TeXButton{End Proof}{\endproof}}%
%BeginExpansion
\endproof
%EndExpansion

\begin{rem}
Formula (\ref{HKLaurent}) is a generalization of the Laurent expansion.
\end{rem}

\begin{rem}
\label{RHardyBall} Let us indicate a relation between the Hardy space
$H_{L}^{2}\left(  \mathcal{A}_{a,b}\right)  $ and analogous space on the ball
$\mathbb{D}\times\mathbb{S}^{d-1}/\mathbb{Z}_{2}$ of the Klein-Dirac quadric
$\operatorname{KDQ}$ given in (\ref{KDQ}). Let the function $f\in H_{L}%
^{2}\left(  \mathcal{A}_{a,b}\right)  $ be polyharmonic  of  order $N$ in the sense
that in representation (\ref{fztheta})%
\[
f\left(  z,\theta\right)  =%
%TCIMACRO{\dsum _{k,\ell}}%
%BeginExpansion
{\displaystyle\sum_{k,\ell}}
%EndExpansion
\left(  \widetilde{f}_{k,\ell}\left(  z^{2}\right)  z^{k}+\widetilde
{\widetilde{f}}_{k,\ell}\left(  z^{2}\right)  z^{-d-k+2}\right)  Y_{k,\ell
}\left(  \theta\right)
\]
the functions $\widetilde{f}_{k,\ell},$ $\widetilde{\widetilde{f}}_{k,\ell}$
are polynomials of degree $N-1.$ Then we have the representation
\[
f\left(  z,\theta\right)  =f_{1}\left(  z,\theta\right)  +z^{-d+2}f_{2}\left(
z,\theta\right)  ,
\]
where the functions
\begin{align*}
f_{1}\left(  z,\theta\right)   &  =%
%TCIMACRO{\dsum _{k,\ell}}%
%BeginExpansion
{\displaystyle\sum_{k,\ell}}
%EndExpansion
\widetilde{f}_{k,\ell}\left(  z^{2}\right)  Y_{k,\ell}\left(  z\theta\right)
\\
f_{2}\left(  z,\theta\right)   &  =%
%TCIMACRO{\dsum _{k,\ell}}%
%BeginExpansion
{\displaystyle\sum_{k,\ell}}
%EndExpansion
\widetilde{\widetilde{f}}_{k,\ell}\left(  z^{2}\right)  Y_{k,\ell}\left(
\frac{\theta}{z}\right)
\end{align*}
a obviously defined on the $\operatorname{KDQ}\setminus\left\{  0\right\}  .$
\end{rem}

\section{The Cauchy type kernel in the component spaces $H^{2,k}\left(
\mathbb{A}_{a,b}\right)  $ \label{ScauchyComponent}}

By the Cauchy formula in the annulus $\mathbb{A}_{a,b},$ for every function
$f$ which is analytic in a neighborhood of $\mathbb{A}_{a,b}$ we have:
\begin{align*}
f\left(  z\right)   &  =\frac{1}{2\pi i}\int_{\Gamma_{b}}\frac{f\left(
\tau\right)  d\tau}{\tau-z}-\frac{1}{2\pi i}\int_{\Gamma_{a}}\frac{f\left(
\tau\right)  d\tau}{\tau-z}\\
&  =\frac{1}{2\pi i}\int_{\Gamma_{b}}\left(  \sum_{j=0}^{\infty}\frac{z^{j}%
}{\tau^{j+1}}\right)  f\left(  \tau\right)  d\tau+\frac{1}{2\pi i}\int
_{\Gamma_{a}}\left(  \sum_{j=0}^{\infty}\frac{\tau^{j}}{z^{j+1}}\right)
f\left(  \tau\right)  d\tau\\
&  =\frac{1}{2\pi}\int_{0}^{2\pi}\left(  \sum_{j=0}^{\infty}\frac{z^{j}}%
{\tau^{j}}\right)  f\left(  \tau\right)  \Big |_{\tau=be^{i\varphi}}d\varphi\\
&  +\frac{1}{2\pi}\int_{0}^{2\pi}\left(  \sum_{j=0}^{\infty}\frac{\tau^{j+1}%
}{z^{j+1}}\right)  f\left(  \tau\right)  \Big |_{\tau=ae^{i\varphi}}d\varphi,
\end{align*}
where we assume that the circular contours $\Gamma_{a}=a\times\mathbb{S}^{1}$
and $\Gamma_{b}=b\times\mathbb{S}^{1}$ have the \textquotedblright
positive\textquotedblright\ (anti-clockwise) orientation and are parametrized
by $\tau=ae^{i\varphi}$ and $\tau=be^{i\varphi};$ we used the notation
$g\left(  \tau\right)  \Big |_{\tau=be^{i\varphi}}=g\left(  be^{i\varphi
}\right)  .$ Hence, we select out of the above series only those terms which
belong to the space (\ref{Vkn}) and we define the following kernels:
\begin{align}
K_{k}^{1}\left(  z,\tau\right)   &  :=\sum_{j=0}^{\infty}\left(  \frac{z}%
{\tau}\right)  ^{k+2j}\qquad\qquad\quad\text{for }\left\vert \tau\right\vert
=b\label{Kk1}\\
K_{k}^{2}\left(  z,\tau\right)   &  :=\sum_{m=-d-k+2+2j\geq0}\left(  \frac
{z}{\tau}\right)  ^{m}\qquad\text{for }\left\vert \tau\right\vert
=b\label{Kk2}\\
K_{k}^{3}\left(  z,\tau\right)   &  :=\sum_{m=-d-k+2+2j<0}\left(  \frac
{z}{\tau}\right)  ^{m}\qquad\text{for }\left\vert \tau\right\vert
=a.\label{Kk3}%
\end{align}
It follows that for every $f\in V_{k,d}$ holds
\begin{align*}
f\left(  z\right)   &  =\frac{1}{2\pi}\int_{0}^{2\pi}\left(  K_{1}\left(
z,\tau\right)  +K_{2}\left(  z,\tau\right)  \right)  f\left(  \tau\right)
\Big |_{\tau=be^{i\varphi}}d\varphi\\
&  +\frac{1}{2\pi}\int_{0}^{2\pi}K_{3}\left(  z,\tau\right)  f\left(
\tau\right)  \Big |_{\tau=ae^{i\varphi}}d\varphi
\end{align*}
Now for every integer $k\geq0$ we will define the \textbf{Cauchy type kernel
for the annulus } $K_{k}$ by setting:
\begin{align}
K_{k}\left(  z,\tau\right)   &  :=K_{k}^{1}\left(  z,\tau\right)  +K_{k}%
^{2}\left(  z,\tau\right)  \qquad\text{for }\left\vert \tau\right\vert
=b\label{Kk}\\
K_{k}\left(  z,\tau\right)   &  :=K_{k}^{3}\left(  z,\tau\right)  \qquad
\qquad\qquad\quad\text{for }\left\vert \tau\right\vert =a.
\end{align}

We have the following Proposition which shows that the function $K_{k}$ is the
Cauchy type kernel for the space $H^{2,k}\left(  \mathbb{A}_{a,b}\right)  $
endowed with the inner product in Remark \ref{RinnerProduct}.

\begin{prop}
\label{PSzego} Let us denote by $f^{\ast}$ the limiting value of the function
$f\in H^{2,k}\left(  \mathbb{A}_{a,b}\right)  $ on $\partial\mathbb{A}_{a,b}%
.$Then for all $z\in\mathbb{A}_{a,b}$ the following Cauchy type formula holds
\begin{align*}
f\left(  z\right)   &  =\frac{1}{2\pi}\int_{0}^{2\pi}\left(  K_{k}^{1}\left(
z,\tau\right)  +K_{k}^{2}\left(  z,\tau\right)  \right)  f^{\ast}\left(
\tau\right)  \Big |_{\tau=be^{i\varphi}}d\varphi\\
&  +\frac{1}{2\pi}\int_{0}^{2\pi}K_{k}^{3}\left(  z,\tau\right)  f^{\ast
}\left(  \tau\right)  \Big |_{\tau=ae^{i\varphi}}d\varphi\\
&  =\frac{1}{2\pi i}%
%TCIMACRO{\dint _{\partial\mathbb{A}_{a,b}}}%
%BeginExpansion
{\displaystyle\int_{\partial\mathbb{A}_{a,b}}}
%EndExpansion
K_{k}\left(  z,\tau\right)  f^{\ast}\left(  \tau\right)  \frac{1}{\tau}d\tau\\
&  =\left\langle K_{k}\left(  z,\cdot\right)  ,f\left(  \cdot\right)
\right\rangle _{H^{2}\left(  \mathbb{A}_{a,b}\right)  }.
\end{align*}

\end{prop}

The proof follows directly from the Cauchy formula for the Hardy space
$H^{2}\left(  \mathbb{D}_{c}\right)  $ (see also Theorem $1$ on p. $9$ in
\cite{sarason}).

A nice feature of the kernel $K_{k}$ is that we have concise expressions:
Obviously, we have
\[
K_{k}^{1}\left(  z,\tau\right)  =\left(  \frac{z}{\tau}\right)  ^{k}\frac
{1}{1-\left(  \frac{z}{\tau}\right)  ^{2}}=\left(  \frac{z}{\tau}\right)
^{k}\frac{\tau^{2}}{\tau^{2}-z^{2}}.
\]
For $k$ \emph{odd} we have the expressions
\begin{align*}
K_{k}^{2}\left(  z,\tau\right)   &  =\sum_{p=0}^{\infty}\left(  \frac{z}{\tau
}\right)  ^{2p}=\frac{1}{1-\left(  \frac{z}{\tau}\right)  ^{2}}=\frac{\tau
^{2}}{\tau^{2}-z^{2}}\qquad\qquad\text{ for }\left\vert \tau\right\vert =b\\
K_{k}^{3}\left(  z,\tau\right)   &  =\sum_{p=1}^{\frac{d+k-2}{2}}\left(
\frac{\tau}{z}\right)  ^{2p}=-\left(  1-\frac{\tau^{d+k-2}}{z^{d+k-2}}\right)
\times K_{k}^{2}\qquad\text{ for }\left\vert \tau\right\vert =a.
\end{align*}
For $k$ \emph{even} we have the expressions
\begin{align*}
K_{k}^{2}\left(  z,\tau\right)   &  =\sum_{p=0}^{\infty}\left(  \frac{z}{\tau
}\right)  ^{2p+1}=\frac{z}{\tau}\frac{\tau^{2}}{\tau^{2}-z^{2}}\qquad
\qquad\qquad\qquad\qquad\text{ for }\left\vert \tau\right\vert =b\\
K_{k}^{3}\left(  z,\tau\right)   &  =\sum_{p=1}^{\frac{d+k-2}{2}}\left(
\frac{\tau}{z}\right)  ^{2p}=-\frac{\tau}{z}\left(  1-\frac{\tau^{d+k-2}%
}{z^{d+k-2}}\right)  \times K_{k}^{2}\left(  z,\tau\right)  \qquad\text{ for
}\left\vert \tau\right\vert =a
\end{align*}

The following estimates are important.

\begin{thm}
\label{TkernelKk} For every $\varepsilon>0$ with $\varepsilon\leq\left(
b-a\right)  /3$ and for $z$ and $\tau$ satisfying $a+\varepsilon<\left\vert
z\right\vert <b-\varepsilon$ and $\left\vert \tau\right\vert =a,$ or
$\left\vert \tau\right\vert =b,$ and for all $k\geq0$ holds
\[
\left\vert K_{k}\left(  z,\tau\right)  \right\vert \leq C_{\varepsilon},
\]
where the constant $C_{\varepsilon}>0$ is independent of $k.$
\end{thm}

The proof follows directly from the above expressions for the kernels
$K_{k}^{j},$ $j=1,2,3.$

\section{Cauchy type kernel for the complexified annulus $\mathcal{A}_{a,b}$
\label{ScauchyHK}}

Finally, we are able to construct the Cauchy kernel for the complexified
annulus $\mathcal{A}_{a,b}.$

First, we prove  the following:

\begin{prop}
Let $f\in H_{L}^{2}\left(  \mathcal{A}_{a,b}\right)  .$ Then in $L_{2}$ sense
the following limits hold true:
\begin{align*}
\lim_{r\rightarrow a}f\left(  re^{i\varphi},\theta\right)   &  =%
%TCIMACRO{\dsum _{k,\ell}}%
%BeginExpansion
{\displaystyle\sum_{k,\ell}}
%EndExpansion
f_{k,\ell}^{\ast}\left(  ae^{i\varphi}\right)  Y_{k,\ell}\left(
\theta\right)  \\
\lim_{r\rightarrow b}f\left(  re^{i\varphi},\theta\right)   &  =%
%TCIMACRO{\dsum _{k,\ell}}%
%BeginExpansion
{\displaystyle\sum_{k,\ell}}
%EndExpansion
f_{k,\ell}^{\ast}\left(  be^{i\varphi}\right)  Y_{k,\ell}\left(
\theta\right)  .
\end{align*}

\end{prop}

The proof mimics the classical proof in \cite{rudin}, Theorem $17.10.$%

%TCIMACRO{\TeXButton{Proof}{\proof} }%
%BeginExpansion
\proof
%EndExpansion
First of all, we see that the function
\[
g\left(  e^{i\varphi},\theta\right)  =%
%TCIMACRO{\dsum _{k,\ell}}%
%BeginExpansion
{\displaystyle\sum_{k,\ell}}
%EndExpansion
f_{k,\ell}^{\ast}\left(  ae^{i\varphi}\right)  Y_{k,\ell}\left(
\theta\right)
\]
belongs to $L_{2}\left(  \mathbb{S}^{1}\times\mathbb{S}^{d-1}\right)  .$
Indeed, we have the estimate
\begin{gather*}%
%TCIMACRO{\dint _{\mathbb{S}^{d-1}}}%
%BeginExpansion
{\displaystyle\int_{\mathbb{S}^{d-1}}}
%EndExpansion%
%TCIMACRO{\dint _{0}^{2\pi}}%
%BeginExpansion
{\displaystyle\int_{0}^{2\pi}}
%EndExpansion
\left\vert
%TCIMACRO{\dsum _{k,\ell}}%
%BeginExpansion
{\displaystyle\sum_{k,\ell}}
%EndExpansion
f_{k,\ell}^{\ast}\left(  ae^{i\varphi}\right)  Y_{k,\ell}\left(
\theta\right)  \right\vert ^{2}d\theta d\varphi\leq%
%TCIMACRO{\dsum _{k,\ell}}%
%BeginExpansion
{\displaystyle\sum_{k,\ell}}
%EndExpansion
\left\Vert f_{k,\ell}^{\ast}\left(  ae^{i\varphi}\right)  \right\Vert
_{L_{2}\left(  \partial\mathbb{A}_{a,b}\right)  }^{2}\\
\leq%
%TCIMACRO{\dsum _{k,\ell}}%
%BeginExpansion
{\displaystyle\sum_{k,\ell}}
%EndExpansion
\left\Vert f_{k,\ell}\right\Vert _{H^{2}\left(  \mathbb{A}_{a,b}\right)  }%
^{2}\leq C%
%TCIMACRO{\dsum _{k,\ell}}%
%BeginExpansion
{\displaystyle\sum_{k,\ell}}
%EndExpansion
\left\Vert f_{k,\ell}\right\Vert _{H^{2}\left(  \mathbb{A}_{a,b}\right)  }%
^{2}L^{2k}\\
\leq C_{1}\left\Vert f\right\Vert _{H_{L}^{2}\left(  \mathcal{A}_{a,b}\right)
}^{2},\qquad\qquad\qquad\qquad\qquad\qquad
\end{gather*}
where $C,C_{1}>0.$ Similar estimate holds for $b.$

Let us put
\[
I_{r}:=%
%TCIMACRO{\dint _{\mathbb{S}^{d-1}}}%
%BeginExpansion
{\displaystyle\int_{\mathbb{S}^{d-1}}}
%EndExpansion%
%TCIMACRO{\dint _{0}^{2\pi}}%
%BeginExpansion
{\displaystyle\int_{0}^{2\pi}}
%EndExpansion
\left\vert
%TCIMACRO{\dsum _{k,\ell}}%
%BeginExpansion
{\displaystyle\sum_{k,\ell}}
%EndExpansion
f_{k,\ell}^{\ast}\left(  ae^{i\varphi}\right)  Y_{k,\ell}\left(
\theta\right)  -%
%TCIMACRO{\dsum _{k,\ell}}%
%BeginExpansion
{\displaystyle\sum_{k,\ell}}
%EndExpansion
f_{k,\ell}\left(  re^{i\varphi}\right)  Y_{k,\ell}\left(  \theta\right)
\right\vert ^{2}d\theta d\varphi.
\]
We have to prove $I_{r}\longrightarrow0$ for $r\longrightarrow a.$ We have
obviously
\[
I_{r}\leq%
%TCIMACRO{\dsum _{k,\ell}}%
%BeginExpansion
{\displaystyle\sum_{k,\ell}}
%EndExpansion%
%TCIMACRO{\dint _{0}^{2\pi}}%
%BeginExpansion
{\displaystyle\int_{0}^{2\pi}}
%EndExpansion
\left\vert f_{k,\ell}^{\ast}\left(  ae^{i\varphi}\right)  -f_{k,\ell}\left(
re^{i\varphi}\right)  \right\vert ^{2}d\varphi.
\]

We fix some $\varepsilon>0.$ There exists a $k_{1}\geq1$ such that for every
$k\geq k_{1}$ holds
\[%
%TCIMACRO{\dsum _{k=k_{1}}^{\infty}}%
%BeginExpansion
{\displaystyle\sum_{k=k_{1}}^{\infty}}
%EndExpansion%
%TCIMACRO{\dsum _{\ell=1}^{a_{k}}}%
%BeginExpansion
{\displaystyle\sum_{\ell=1}^{a_{k}}}
%EndExpansion
\left\Vert f_{k,\ell}\right\Vert _{H^{2}\left(  \mathbb{A}_{a,b}\right)  }%
^{2}L^{2k}<\frac{\varepsilon}{2}.
\]
This implies
\[%
%TCIMACRO{\dsum _{k=k_{1}}^{\infty}}%
%BeginExpansion
{\displaystyle\sum_{k=k_{1}}^{\infty}}
%EndExpansion%
%TCIMACRO{\dsum _{\ell=1}^{a_{k}}}%
%BeginExpansion
{\displaystyle\sum_{\ell=1}^{a_{k}}}
%EndExpansion%
%TCIMACRO{\dint _{0}^{2\pi}}%
%BeginExpansion
{\displaystyle\int_{0}^{2\pi}}
%EndExpansion
\left\vert f_{k,\ell}^{\ast}\left(  ae^{i\varphi}\right)  -f_{k,\ell}\left(
re^{i\varphi}\right)  \right\vert ^{2}d\varphi<C\frac{\varepsilon}{2}
\]
for some constant $C>0.$ On the other hand, we find $\delta>0$ such that for
$\left\vert r-a\right\vert <\delta$ holds
\[%
%TCIMACRO{\dsum _{k=0}^{k_{1}-1}}%
%BeginExpansion
{\displaystyle\sum_{k=0}^{k_{1}-1}}
%EndExpansion%
%TCIMACRO{\dsum _{\ell=1}^{a_{k}}}%
%BeginExpansion
{\displaystyle\sum_{\ell=1}^{a_{k}}}
%EndExpansion%
%TCIMACRO{\dint _{0}^{2\pi}}%
%BeginExpansion
{\displaystyle\int_{0}^{2\pi}}
%EndExpansion
\left\vert f_{k,\ell}^{\ast}\left(  ae^{i\varphi}\right)  -f_{k,\ell}\left(
re^{i\varphi}\right)  \right\vert ^{2}d\varphi<C\frac{\varepsilon}{2}.
\]
Hence, we obtain $I_{r}\leq C\varepsilon$ which proves that $I_{r}%
\longrightarrow0$ for $r\longrightarrow a.$%

%TCIMACRO{\TeXButton{End Proof}{\endproof}}%
%BeginExpansion
\endproof
%EndExpansion

Recall that for every $f\in H_{L}^{2}\left(  \mathcal{A}_{a,b}\right)  $ all
components $f_{k,\ell}\in H^{2,k}\left(  \mathbb{A}_{a,b}\right)  $ and they
have limiting functions $f_{k,\ell}^{\ast}\in L_{2}\left(  \partial
\mathbb{A}_{a,b}\right)  .$ Hence, by Proposition \ref{PSzego} we obtain
\[
f\left(  z,\theta\right)  =%
%TCIMACRO{\dsum _{k,\ell}}%
%BeginExpansion
{\displaystyle\sum_{k,\ell}}
%EndExpansion
\left\langle K_{k}\left(  z,\cdot\right)  ,f_{k,\ell}^{\ast}\left(
\cdot\right)  \right\rangle _{H^{2}\left(  \mathbb{A}_{a,b}\right)  }%
Y_{k,\ell}\left(  \theta\right)  \qquad\text{for all } \left(z,\theta\right) \in
\mathcal{A}_{a,b}.
\]

\begin{defn}
\label{DcauchyTypeKernel} We define the\textbf{\ Cauchy type kernel} $K$ for
the space $H_{L}^{2}\left(  \mathcal{A}_{a,b}\right)  $ by putting
\begin{equation}
K\left(  z,\theta;\tau,\theta^{\prime}\right)  =\sum_{k=0}^{\infty}\sum
_{\ell=1}^{a_{k}}\frac{1}{L^{k}}K_{k}\left(  z,\tau\right)  Y_{k,\ell}\left(
\theta\right)  Y_{k,\ell}\left(  \theta^{\prime}\right)  . \label{kernelAll}%
\end{equation}

\end{defn}

The function $K\left(  z,\theta;\tau,\theta^{\prime}\right)  $ converges on
compact subsets of $\mathcal{A}_{a,b}$: Indeed, we use the estimates
$\left\vert Y_{k,\ell}\left(  \theta\right)  \right\vert \leq C_{1}%
k^{\frac{d-2}{2}},$ and $a_{k}\leq C_{2}k^{d-2},$ cf. \cite{seeley}. By
Theorem \ref{TkernelKk} we obtain the estimate
\[
\left\vert K\left(  z,\theta;\tau,\theta^{\prime}\right)  \right\vert \leq
C_{3,\varepsilon}\sum_{k=0}^{\infty}\frac{k^{d-2}k^{d-2}}{L^{k}}<\infty
\]
for all $a+\varepsilon\leq\left\vert z\right\vert \leq b-\varepsilon$ and
$\left\vert \tau\right\vert =a$ or $\left\vert \tau\right\vert =b.$

Finally, we generalize the classical one-dimensional Cauchy formula, and
justify the name "Cauchy type kernel" given to $K\left(  z,\theta;\tau
,\theta^{\prime}\right)  $ in Definition \ref{DcauchyTypeKernel}.

\begin{thm}
The kernel $K\left(  \zeta,\theta;z,\theta^{\prime}\right)  $ is a Cauchy type
kernel for the space $H_{L}^{2}\left(  \mathcal{A}_{a,b}\right)  $ defined by
the inner product (\ref{HKspace}), i.e. for every $f\in H_{L}^{2}\left(
\mathcal{A}_{a,b}\right)  $ holds
\begin{equation}
\left\langle K\left(  z,\theta;z^{\prime},\theta^{\prime}\right)  ,f^{\ast
}\left(  z^{\prime},\theta^{\prime}\right)  \right\rangle _{H_{L}^{2}%
}=f\left(  z,\theta\right)  ,\qquad\text{for all }\left(  z,\theta\right)
\in\mathcal{A}_{a,b}, \label{reprod}%
\end{equation}
where $f^{\ast}$ is the boundary limit at $\partial\mathcal{A}_{a,b}$ of the
function $f.$
\end{thm}

The proof follows by a direct application of Theorem \ref{TkernelKk}.

\section{Error estimate of the Polyharmonic Gauss-Jacobi Cubature formula in
the annulus \label{Scubature}}

The topic of estimation of quadrature formulas for analytic functions is a
widely studied one. Beyond the classical monographs \cite{krylov},
\cite{davisRabinowitz}, we provide further and more recent publications, as
\cite{bakhvalov}, \cite{gautschi}, \cite{goetz}, \cite{kzaz}, \cite{kowalski},
\cite{milova}. No references may be found though for the the Numerical
Integration in the multivariate case, which is often called cubature formulas,
even in the fundamental monographs as \cite{sobolev}, \cite{stroudBook},
\cite{sobolev2}; see also the recent survey \cite{cools}.

In order to explain our approach, first of all, we will recall the
\emph{polyharmonic Gauss-Jacobi cubature} \emph{formula} presented in
\cite{kounchevRenderArxiv}, \cite{kounchevRenderArkiv} for the annulus
$A_{a,b}\subset\mathbb{R}^{d}.$ Let us fix some integers $N\geq1$ and $k\geq0.
$ We consider the following $2N-$dimensional subspace of $V_{k,d}$ in
(\ref{Vkn}),
\[
V_{k,d,N}:=\operatorname*{span}\left\{  \left\{  r^{2j+k}\right\}
_{j=0}^{N-1},\ \left\{  r^{-d-k+2+2j}\right\}  _{j=0}^{N-1}\right\}  .
\]
We remark that this is a Chebyshev system of order $2N,$ cf.
\cite{kreinNudelman}.

Let the \emph{pseudo-positive} (signed) measure $d\mu$ be given in the annulus
$A_{a,b},$ see (\ref{defpspos}). For all indices $\left(  k,\ell\right)  $ the
component measures are defined by
\begin{equation}
d\mu_{k,\ell}\left(  r\right)  :=%
%TCIMACRO{\dint _{\mathbb{S}^{d-1}}}%
%BeginExpansion
{\displaystyle\int_{\mathbb{S}^{d-1}}}
%EndExpansion
Y_{k,\ell}\left(  \theta\right)  d\mu\left(  r\theta\right)  \geq
0\qquad\text{for all }r\in\left[  a,b\right]  ; \label{dmukl}%
\end{equation}
here the integral is symbolical with respect to the variables $\theta.$
Rigorously, the component measure $d\mu_{k,\ell}\left(  r\right)  $ is defined
for the functions $g\left(  r\right)  $ on the interval $\left[  a,b\right]  $
by means of the equality
\[%
%TCIMACRO{\dint _{a}^{b}}%
%BeginExpansion
{\displaystyle\int_{a}^{b}}
%EndExpansion
g\left(  r\right)  d\mu_{k,\ell}\left(  r\right)  :=%
%TCIMACRO{\dint _{A_{a,b}}}%
%BeginExpansion
{\displaystyle\int_{A_{a,b}}}
%EndExpansion
g\left(  r\right)  Y_{k,\ell}\left(  \theta\right)  d\mu\left(  x\right)  ;
\]
cf. \cite{kounchevRenderArxiv}, \cite{kounchevRenderArkiv}.

Let us fix $\left(  k,\ell\right)  .$ Recall from \cite{kounchevRenderArxiv},
\cite{kounchevRenderArkiv}, that we consider truncated Moment problem and
related Gauss-Jacobi type Quadrature formula corresponding to the operator
$\Delta^{2N}.$ Thus, there exist points $t_{k,\ell;j},$ $j=1,2,...,2N,$
belonging to the interval $\left[  a,b\right]  ,$ and non-negative numbers
$\left\{  \lambda_{k,\ell;j}\right\}  _{j=1}^{2N},$ such that the following
quadrature formula holds:
\begin{align}%
%TCIMACRO{\dint _{a}^{b}}%
%BeginExpansion
{\displaystyle\int_{a}^{b}}
%EndExpansion
Q\left(  t\right)  d\mu_{k,\ell}\left(  t\right)   &  =%
%TCIMACRO{\dint _{a}^{b}}%
%BeginExpansion
{\displaystyle\int_{a}^{b}}
%EndExpansion
Q\left(  t\right)  d\mu_{k,\ell}^{G}\left(  t\right) \label{Quadraturekl}\\
&  =%
%TCIMACRO{\dsum _{j=1}^{N}}%
%BeginExpansion
{\displaystyle\sum_{j=1}^{N}}
%EndExpansion
\lambda_{k,\ell;j}Q\left(  t_{k,\ell;j}\right)  \qquad\text{for every }Q\in
V_{k,d,2N},\nonumber
\end{align}
cf. \cite{kreinNudelman}, Theorem $4.1,$ chapter $4.$ We have put
\begin{equation}
d\mu_{k,\ell}^{G}=%
%TCIMACRO{\dsum _{j=1}^{N}}%
%BeginExpansion
{\displaystyle\sum_{j=1}^{N}}
%EndExpansion
\lambda_{k,\ell;j}\delta\left(  t-t_{k,\ell;j}\right)  . \label{dmuG}%
\end{equation}

First of all, we will find the error for the above Quadrature formulas
(\ref{Quadraturekl}) following the classical scheme outlined in
\cite{davisRabinowitz}, p. $231-235$ (see also chapter $12$ in \cite{krylov}).

Let $H_{k,\ell;2N}\left[  f\right]  \left(  t\right)  $ be the unique element
of $V_{k,d,N}$ which interpolates the function $f$ at the points $t_{k,\ell
;j}, $ i.e.
\begin{equation}
H_{k,\ell;2N}\left[  f\right]  \left(  t_{k,\ell;j}\right)  =f\left(
t_{k,\ell;j}\right)  \qquad\text{for }j=1,2,...,2N. \label{Hkl2Ninterpolate}%
\end{equation}
On the other hand, for every function $f\in H^{2,k}\left(  \mathbb{A}%
_{a,b}\right)  $ by Proposition \ref{PSzego} we have the Cauchy type formula
\[
f\left(  z\right)  =\frac{1}{2\pi i}%
%TCIMACRO{\dint _{\partial\mathbb{A}_{a,b}}}%
%BeginExpansion
{\displaystyle\int_{\partial\mathbb{A}_{a,b}}}
%EndExpansion
K_{k}\left(  z,\tau\right)  f^{\ast}\left(  \tau\right)  \frac{1}{\tau}d\tau,
\]
where $f^{\ast}$ is the limiting value of $f$ on the boundary $\partial
\mathbb{A}_{a,b}.$ Hence, we obtain
%% comment  4.6.13 in Davis-Rabin%
\[
H_{k,\ell;2N}\left[  f\right]  \left(  z\right)  =\frac{1}{2\pi i}%
%TCIMACRO{\dint _{\partial\mathbb{A}_{a,b}}}%
%BeginExpansion
{\displaystyle\int_{\partial\mathbb{A}_{a,b}}}
%EndExpansion
H_{k,\ell;2N}\left[  K_{k}\left(  \cdot,\tau\right)  \right]  \left(
z\right)  f^{\ast}\left(  \tau\right)  \frac{1}{\tau}d\tau;
\]
here for every fixed $\tau,$ by $H_{k,\ell;2N}\left[  K_{k}\left(  \cdot
,\tau\right)  \right]  \left(  z\right)  $ we denote the element in
$V_{k,d,N}$ which interpolates the kernel function $K_{k}\left(
z,\tau\right)  $ at $z=t_{k,\ell;j}$ for $j=1,2,...,2N.$ This gives us the
remainder formula of Cauchy type for the interpolation (\ref{Hkl2Ninterpolate}%
),
\begin{align}
f\left(  z\right)   &  =H_{k,\ell;2N}\left[  f\right]  \left(  z\right)
\label{fklRemainder}\\
&  +\frac{1}{2\pi i}%
%TCIMACRO{\dint _{\partial\mathbb{A}_{a,b}}}%
%BeginExpansion
{\displaystyle\int_{\partial\mathbb{A}_{a,b}}}
%EndExpansion
\left\{  K_{k}\left(  z,\tau\right)  -H_{k,\ell;2N}\left[  K_{k}\left(
\cdot,\tau\right)  \right]  \left(  z\right)  \right\}  f^{\ast}\left(
\tau\right)  \frac{1}{\tau}d\tau.\nonumber
\end{align}

Now we are able to estimate the error functional:
%% comment  4.6.17 in Davis-Rabincomment%
\begin{equation}
E_{k,\ell}\left[  f\right]  :=%
%TCIMACRO{\dint _{a}^{b}}%
%BeginExpansion
{\displaystyle\int_{a}^{b}}
%EndExpansion
f\left(  t\right)  d\mu_{k,\ell}\left(  t\right)  -%
%TCIMACRO{\dint _{a}^{b}}%
%BeginExpansion
{\displaystyle\int_{a}^{b}}
%EndExpansion
f\left(  t\right)  d\mu_{k,\ell}^{G}\left(  t\right)  . \label{Errorf}%
\end{equation}
Indeed, by (\ref{fklRemainder}) we have
\begin{gather*}%
%TCIMACRO{\dint _{a}^{b}}%
%BeginExpansion
{\displaystyle\int_{a}^{b}}
%EndExpansion
f\left(  z\right)  d\mu_{k,\ell}\left(  z\right)  =%
%TCIMACRO{\dint _{a}^{b}}%
%BeginExpansion
{\displaystyle\int_{a}^{b}}
%EndExpansion
H_{k,\ell;2N}\left[  f\right]  \left(  z\right)  d\mu_{k,\ell}\left(
z\right)  \qquad\qquad\qquad\qquad\qquad\qquad\qquad\qquad\\
\qquad+\frac{1}{2\pi i}%
%TCIMACRO{\dint _{\partial\mathbb{A}_{a,b}}}%
%BeginExpansion
{\displaystyle\int_{\partial\mathbb{A}_{a,b}}}
%EndExpansion%
%TCIMACRO{\dint _{a}^{b}}%
%BeginExpansion
{\displaystyle\int_{a}^{b}}
%EndExpansion
\left\{  K_{k}\left(  z,\tau\right)  -H_{k,\ell;2N}\left[  K_{k}\left(
\cdot,\tau\right)  \right]  \left(  z\right)  \right\}  d\mu_{k,\ell}\left(
z\right)  f^{\ast}\left(  \tau\right)  \frac{1}{\tau}d\tau
\end{gather*}
On the other hand, by (\ref{Quadraturekl}) and (\ref{Hkl2Ninterpolate}) we
obtain
\begin{align*}%
%TCIMACRO{\dint _{a}^{b}}%
%BeginExpansion
{\displaystyle\int_{a}^{b}}
%EndExpansion
H_{k,\ell;2N}\left[  f\right]  \left(  z\right)  d\mu_{k,\ell}\left(
z\right)   &  =%
%TCIMACRO{\dint _{a}^{b}}%
%BeginExpansion
{\displaystyle\int_{a}^{b}}
%EndExpansion
H_{k,\ell;2N}\left[  f\right]  \left(  z\right)  d\mu_{k,\ell}^{G}\left(
z\right) \\
&  =%
%TCIMACRO{\dint _{a}^{b}}%
%BeginExpansion
{\displaystyle\int_{a}^{b}}
%EndExpansion
f\left(  z\right)  d\mu_{k,\ell}^{G}\left(  z\right)
\end{align*}
which implies
%% comment  4.6.18  in Davis-Rabincomment%
\begin{align}
E_{k,\ell}\left[  f\right]   &  =%
%TCIMACRO{\dint _{a}^{b}}%
%BeginExpansion
{\displaystyle\int_{a}^{b}}
%EndExpansion
f\left(  t\right)  d\mu_{k,\ell}\left(  t\right)  -%
%TCIMACRO{\dint _{a}^{b}}%
%BeginExpansion
{\displaystyle\int_{a}^{b}}
%EndExpansion
f\left(  t\right)  d\mu_{k,\ell}^{G}\left(  t\right) \label{Eklf=}\\
&  =\frac{1}{2\pi i}%
%TCIMACRO{\dint _{\partial\mathbb{A}_{a,b}}}%
%BeginExpansion
{\displaystyle\int_{\partial\mathbb{A}_{a,b}}}
%EndExpansion%
%TCIMACRO{\dint _{a}^{b}}%
%BeginExpansion
{\displaystyle\int_{a}^{b}}
%EndExpansion
\left\{  K_{k}\left(  z,\tau\right)  -H_{k,\ell;2N}\left[  K_{k}\left(
\cdot,\tau\right)  \right]  \left(  z\right)  \right\}  d\mu_{k,\ell}\left(
z\right)  f^{\ast}\left(  \tau\right)  \frac{1}{\tau}d\tau\nonumber
\end{align}
and we see that in the right-hand side we do not have derivatives of the
function $f$ but only its values on the two circles $\partial\mathbb{A}%
_{a,b}.$ Now we may prove the following Lemma.

\begin{lem}
\label{LCk} Let the knots $\left\{  t_{k,\ell;j}\right\}  _{j=1}^{2N}$ of the
quadrature formula (\ref{Quadraturekl}) lie in the interval $\left[
a,b\right]  \subset\left(  a^{\prime},b^{\prime}\right)  .$ Then for every
function $f$ in the component Hardy space $H^{2,k}\left(  \mathbb{A}%
_{a^{\prime},b^{\prime}}\right)  ,$ the error of the quadrature formula
(\ref{Quadraturekl}) satisfies
\begin{equation}
\left\vert E_{k,\ell}\left[  f\right]  \right\vert =\left\vert
%TCIMACRO{\dint _{a}^{b}}%
%BeginExpansion
{\displaystyle\int_{a}^{b}}
%EndExpansion
f\left(  t\right)  d\mu_{k,\ell}\left(  t\right)  -%
%TCIMACRO{\dint _{a}^{b}}%
%BeginExpansion
{\displaystyle\int_{a}^{b}}
%EndExpansion
f\left(  t\right)  d\mu_{k,\ell}^{G}\left(  t\right)  \right\vert \leq
C_{k}\left\Vert f\right\Vert _{H^{2}\left(  \mathbb{A}_{a^{\prime},b^{\prime}%
}\right)  }. \label{Ck}%
\end{equation}

\end{lem}

%

%TCIMACRO{\TeXButton{Proof}{\proof} }%
%BeginExpansion
\proof
%EndExpansion
Since $\left\{  t_{k,\ell;j}\right\}  _{j=1}^{2N}\in\left[  a,b\right]
\subset\left(  a^{\prime},b^{\prime}\right)  $ it follows that the
interpolation operator $H_{k,\ell;2N}\left[  K_{k}\left(  \cdot,\tau\right)
\right]  $ depends continuously on $\tau\in\partial\mathbb{A}_{a^{\prime
},b^{\prime}}.$ Hence, the function
\[%
%TCIMACRO{\dint _{a}^{b}}%
%BeginExpansion
{\displaystyle\int_{a}^{b}}
%EndExpansion
\left\{  K_{k}\left(  z,\tau\right)  -H_{k,\ell;2N}\left[  K_{k}\left(
\cdot,\tau\right)  \right]  \left(  z\right)  \right\}  d\mu_{k,\ell}\left(
z\right)
\]
is also continuous in $\tau\in\partial\mathbb{A}_{a^{\prime},b^{\prime}}$ and
by Theorem \ref{TkernelKk} its modulus is bounded, namely
%% comment  4.6.19 in Davis-Rabin  comment%
\[
\left\vert
%TCIMACRO{\dint _{a}^{b}}%
%BeginExpansion
{\displaystyle\int_{a}^{b}}
%EndExpansion
\left\{  K_{k}\left(  z,\tau\right)  -H_{k,\ell;2N}\left[  K_{k}\left(
\cdot,\tau\right)  \right]  \left(  z\right)  \right\}  d\mu_{k,\ell}\left(
z\right)  \right\vert \leq C_{k}\qquad\text{for }\tau\in\partial
\mathbb{A}_{a^{\prime},b^{\prime}}.
\]
By application of Cauchy-Bunyakovksi-Schwarz inequality, from (\ref{Eklf=}),
we obtain for some constant $C>0$ the inequality
\[
\left\vert
%TCIMACRO{\dint }%
%BeginExpansion
{\displaystyle\int}
%EndExpansion
f\left(  t\right)  d\mu_{k,\ell}\left(  t\right)  -%
%TCIMACRO{\dint }%
%BeginExpansion
{\displaystyle\int}
%EndExpansion
f\left(  t\right)  d\mu_{k,\ell}^{G}\left(  t\right)  \right\vert \leq
CC_{k}\left\Vert f\right\Vert _{H^{2}\left(  \mathbb{A}_{a^{\prime},b^{\prime
}}\right)  }%
\]
which ends the proof.%

%TCIMACRO{\TeXButton{End Proof}{\endproof}}%
%BeginExpansion
\endproof
%EndExpansion

Finally, we may proceed to the \emph{polyharmonic Gauss-Jacobi Cubature
formula} for the pseudo-positive measure $d\mu.$ We will consider this in a
simpler setting, for functions $f$ having a \textbf{finite} Laplace-Fourier
expansion:
\begin{equation}
f\left(  x\right)  =f\left(  r\theta\right)  =%
%TCIMACRO{\dsum _{k,\ell,\ k\leq k_{0}}}%
%BeginExpansion
{\displaystyle\sum_{k,\ell,\ k\leq k_{0}}}
%EndExpansion
f_{k,\ell}\left(  r\right)  Y_{k,\ell}\left(  \theta\right)  . \label{ffinite}%
\end{equation}
The polyharmonic Gauss-Jacobi measure $d\mu^{G}\left(  x\right)  $ is defined
as the pseudo-positive measure having component measures (\ref{dmukl}) equal
to $d\mu_{k,\ell}^{G}$ in (\ref{dmuG}). Written symbolically, we have the
expansion
\begin{equation}
d\mu^{G}\left(  x\right)  =%
%TCIMACRO{\dsum _{k,\ell}}%
%BeginExpansion
{\displaystyle\sum_{k,\ell}}
%EndExpansion
Y_{k,\ell}\left(  \theta\right)  d\mu_{k,\ell}^{G}\left(  r\right)  .
\label{dmuGAll}%
\end{equation}
The corresponding Cubature formula is simple to formulate if the function $f$
defined in $A_{a,b}$ has a finite expansion as in (\ref{ffinite}):
\begin{equation}%
%TCIMACRO{\dint _{A_{a,b}}}%
%BeginExpansion
{\displaystyle\int_{A_{a,b}}}
%EndExpansion
f\left(  x\right)  d\mu^{G}\left(  x\right)  =%
%TCIMACRO{\dsum _{k,\ell,\ k\leq k_{0}}}%
%BeginExpansion
{\displaystyle\sum_{k,\ell,\ k\leq k_{0}}}
%EndExpansion%
%TCIMACRO{\dint _{a}^{b}}%
%BeginExpansion
{\displaystyle\int_{a}^{b}}
%EndExpansion
f_{k,\ell}\left(  r\right)  d\mu_{k,\ell}^{G}\left(  r\right)  .
\label{CubatureAll}%
\end{equation}
It is direct to see from the quadrature formulas (\ref{Quadraturekl}), that
for a function $u\left(  x\right)  $ which is a finite sum of the type
(\ref{uxAlmansiAnnulus}) with $\Delta^{2N}u\left(  x\right)  =0$ in $A_{a,b},$
holds
\[%
%TCIMACRO{\dint _{A_{a,b}}}%
%BeginExpansion
{\displaystyle\int_{A_{a,b}}}
%EndExpansion
u\left(  x\right)  d\mu\left(  x\right)  =%
%TCIMACRO{\dint _{A_{a,b}}}%
%BeginExpansion
{\displaystyle\int_{A_{a,b}}}
%EndExpansion
u\left(  x\right)  d\mu^{G}\left(  x\right)  ,
\]
cf. \cite{kounchevRenderArxiv}, \cite{kounchevRenderArkiv}. This justifies the
name \emph{polyharmonic cubature formula}. Hence, we define the Cubature
formula by putting
\begin{equation}
C_{N}\left(  f\right)  :=%
%TCIMACRO{\dint _{A_{a,b}}}%
%BeginExpansion
{\displaystyle\int_{A_{a,b}}}
%EndExpansion
f\left(  x\right)  d\mu^{G}\left(  x\right)  . \label{CNdefined}%
\end{equation}

We will prove the estimate of the cubature formula in the following.

\begin{thm}
\label{TcubatureEstimate} Let $d\mu$ be a pseudo-positive measure, and
$C_{N}\left(  f\right)  $ be the polyharmonic Gauss-Jacobi Cubature formula
defined by (\ref{CNdefined}) in the annulus $A_{a,b}.$ Let $a^{\prime}<a$ and
$b<b^{\prime},$ and we assume that the function $f\in H_{L}^{2}\left(
\mathcal{A}_{a^{\prime},b^{\prime}}\right)  $ and has a finite representation
(\ref{ffinite}) for some $k_{0}\geq0.$

Then the error of the Cubature formula (\ref{CubatureAll}) is estimated by
\begin{align*}
\left\vert E\left[  f\right]  \right\vert  &  =\left\vert
%TCIMACRO{\dint _{A_{a,b}}}%
%BeginExpansion
{\displaystyle\int_{A_{a,b}}}
%EndExpansion
f\left(  x\right)  d\mu\left(  x\right)  -C_{N}\left(  f\right)  \right\vert
\\
&  \leq\sqrt{%
%TCIMACRO{\dsum _{k,\ell,\ k\leq k_{0}}}%
%BeginExpansion
{\displaystyle\sum_{k,\ell,\ k\leq k_{0}}}
%EndExpansion
\frac{C_{k}^{2}}{L^{2k}}}\times\left\Vert f\right\Vert _{H_{L}^{2}\left(
\mathcal{A}_{a^{\prime},b^{\prime}}\right)  }^{2}.
\end{align*}

\end{thm}

%

%TCIMACRO{\TeXButton{Proof}{\proof} }%
%BeginExpansion
\proof
%EndExpansion
Due to the finite expansion, we obtain the equality
\[%
%TCIMACRO{\dint _{A_{a,b}}}%
%BeginExpansion
{\displaystyle\int_{A_{a,b}}}
%EndExpansion
f\left(  x\right)  d\mu\left(  x\right)  =%
%TCIMACRO{\dsum _{k,\ell,\ k\leq k_{0}}}%
%BeginExpansion
{\displaystyle\sum_{k,\ell,\ k\leq k_{0}}}
%EndExpansion%
%TCIMACRO{\dint _{a}^{b}}%
%BeginExpansion
{\displaystyle\int_{a}^{b}}
%EndExpansion
f_{k,\ell}\left(  r\right)  d\mu_{k,\ell}\left(  r\right)  .
\]
On the other hand, by the Cubature formula (\ref{CubatureAll}) we have
\[%
%TCIMACRO{\dint _{A_{a,b}}}%
%BeginExpansion
{\displaystyle\int_{A_{a,b}}}
%EndExpansion
f\left(  x\right)  d\mu^{G}\left(  x\right)  =%
%TCIMACRO{\dsum _{k,\ell,\ k\leq k_{0}}}%
%BeginExpansion
{\displaystyle\sum_{k,\ell,\ k\leq k_{0}}}
%EndExpansion%
%TCIMACRO{\dint _{a}^{b}}%
%BeginExpansion
{\displaystyle\int_{a}^{b}}
%EndExpansion
f_{k,\ell}\left(  r\right)  d\mu_{k,\ell}^{G}\left(  r\right)  .
\]
Hence, we obtain the error functional
\begin{align*}
E\left[  f\right]   &  =%
%TCIMACRO{\dint _{A_{a,b}}}%
%BeginExpansion
{\displaystyle\int_{A_{a,b}}}
%EndExpansion
f\left(  x\right)  d\mu\left(  x\right)  -%
%TCIMACRO{\dint _{A_{a,b}}}%
%BeginExpansion
{\displaystyle\int_{A_{a,b}}}
%EndExpansion
f\left(  x\right)  d\mu^{G}\left(  x\right) \\
&  =%
%TCIMACRO{\dsum _{k,\ell,\ k\leq k_{0}}}%
%BeginExpansion
{\displaystyle\sum_{k,\ell,\ k\leq k_{0}}}
%EndExpansion
\left(
%TCIMACRO{\dint _{a}^{b}}%
%BeginExpansion
{\displaystyle\int_{a}^{b}}
%EndExpansion
f_{k,\ell}\left(  r\right)  d\mu_{k,\ell}\left(  r\right)  -%
%TCIMACRO{\dint _{a}^{b}}%
%BeginExpansion
{\displaystyle\int_{a}^{b}}
%EndExpansion
f_{k,\ell}\left(  r\right)  d\mu_{k,\ell}^{G}\left(  r\right)  \right)
\end{align*}
and by Lemma \ref{LCk} we obtain
\begin{align*}
\left\vert E\left[  f\right]  \right\vert  &  =\left\vert
%TCIMACRO{\dint _{A_{a,b}}}%
%BeginExpansion
{\displaystyle\int_{A_{a,b}}}
%EndExpansion
f\left(  x\right)  d\mu\left(  x\right)  -%
%TCIMACRO{\dint _{A_{a,b}}}%
%BeginExpansion
{\displaystyle\int_{A_{a,b}}}
%EndExpansion
f\left(  x\right)  d\mu^{G}\left(  x\right)  \right\vert \\
&  \leq%
%TCIMACRO{\dsum _{k,\ell,\ k\leq k_{0}}}%
%BeginExpansion
{\displaystyle\sum_{k,\ell,\ k\leq k_{0}}}
%EndExpansion
C_{k}\left\Vert f_{k,\ell}\right\Vert _{H^{2}\left(  \mathbb{A}_{a^{\prime
},b^{\prime}}\right)  }=%
%TCIMACRO{\dsum _{k,\ell,\ k\leq k_{0}}}%
%BeginExpansion
{\displaystyle\sum_{k,\ell,\ k\leq k_{0}}}
%EndExpansion
\frac{C_{k}}{L^{k}}\left\Vert f_{k,\ell}\right\Vert _{H^{2}\left(
\mathbb{A}_{a^{\prime},b^{\prime}}\right)  }L^{k}\\
&  \leq\sqrt{%
%TCIMACRO{\dsum _{k,\ell,\ k\leq k_{0}}}%
%BeginExpansion
{\displaystyle\sum_{k,\ell,\ k\leq k_{0}}}
%EndExpansion
\frac{C_{k}^{2}}{L^{2k}}}\times\left\Vert f\right\Vert _{H_{L}^{2}\left(
\mathcal{A}_{a^{\prime},b^{\prime}}\right)  }^{2}.
\end{align*}
This ends the proof.%

%TCIMACRO{\TeXButton{End Proof}{\endproof}}%
%BeginExpansion
\endproof
%EndExpansion

For practical purposes it is reasonable to work with functions having finite
expansion (\ref{ffinite}). However, from Theorem \ref{TcubatureEstimate} we
see that the asymptotic behavior of the constants $C_{k}$ defined in
(\ref{Ck}) is important for the estimate of the error functional $E\left[
f\right]  $ for functions $f$ which do not have a finite expansion
(\ref{ffinite}). This is a subtle task of Interpolation theory which is beyond
the scope of the present paper.

Another important remark is that the cubature formula (\ref{CNdefined}) may be
easily extended, although in a non-unique manner, to signed measures $\mu$
having bounded variation. Indeed, we may represent the measure $\mu$ as a
difference, $\mu=\mu^{1}-\mu^{2},$ where $\mu^{1},$ $\mu^{2}$ have bounded
variation. We can do this representation componentwise (non-uniquely) by
$\mu_{k,\ell}=\mu_{k,\ell}^{1}-\mu_{k,\ell}^{2}$ with measures $d\mu_{k,\ell
}^{1},d\mu_{k,\ell}^{2}\geq0.$ Then we find the polyharmonic Gau\ss -Jacobi
cubature measures $\mu^{1,G}$ and $\mu^{2,G}$ and put $C_{N}\left(  f\right)
=%
%TCIMACRO{\dint }%
%BeginExpansion
{\displaystyle\int}
%EndExpansion
fd\mu^{1,G}-%
%TCIMACRO{\dint }%
%BeginExpansion
{\displaystyle\int}
%EndExpansion
fd\mu^{2,G}$ for the cubature formula.

Our research may be considered as a contribution to the topic of analytic
continuation of solutions to elliptic equations (in particular, harmonic
functions), see the discussion and references to the works of V. Avanissian,
P. Lelong, C. Kiselman, J. Siciak, M. Jarnicki, T. du Cros, on p. $54-55$ in
\cite{aron}, \cite{avanissian}, \cite{Mori98}, \cite{FuMo02}, \cite{KoRe08},
\cite{KounchevRenderPoly}. On the other hand, the concept of
\emph{polyharmonic Hardy spaces} which we introduce appears to be a new
multivariate concept of Hardy space which differs from the existing
approaches, cf. \cite{stein}, \cite{steinWeiss}, \cite{rudin70},
\cite{rudin80}, \cite{coifman}, \cite{sarason1998}, \cite{Shai03}. This notion
will be given a thorough study in a planned monograph
\cite{kounchevRenderBook}. Another setting where the polyharmonic Gauss-Jacobi
cubature formulas and Hardy space $H_{L}^{2}$ may be naturally developed, is
the strip $\left[  a,b\right]  \times\mathbb{R}^{d},$ cf.
\cite{kounchevRenderArxiv}, \cite{kounchevRenderBook}.

\subsection*{Acknowledgment}

Both authors thank the Alexander von Humboldt Foundation.

\end{document}